\pdfoutput=1
\documentclass[11pt]{article}

\usepackage[a4paper, total={6in, 8in}]{geometry}
\usepackage[english]{babel}

\usepackage{graphicx}
\usepackage{wrapfig}
\usepackage{caption}
\usepackage{subcaption}
\usepackage{xcolor}

\usepackage{cite}
\usepackage{comment}

\usepackage{amsmath,amsthm,amssymb,mathtools}
\usepackage{esvect}
\usepackage{hyperref}

\usepackage{float}
\usepackage{algorithm}
\usepackage{algpseudocode}

\usepackage{enumitem}

\usepackage{titlesec}
\titleformat*{\paragraph}{\itshape}
\titleformat*{\subsubsection}{\itshape}

\renewcommand{\P}{\mathrm{P}}
\newcommand{\Z}{\mathbb{Z}}
\newcommand{\A}{\mathcal{A}}

\DeclareMathOperator{\br}{br}
\DeclareMathOperator{\cl}{cl}

\DeclarePairedDelimiter{\abs}{\lvert}{\rvert}

\newtheorem{theorem}{Theorem}
\newtheorem{corollary}{Corollary}
\newtheorem{lemma}{Lemma}

\newtheorem{example}{Example}

\theoremstyle{remark}
\newtheorem{remark}{Remark}

\usepackage[bottom]{footmisc}

\title{Critical probabilities for positively associated, finite-range dependent percolation models}
\author{Laurin Köhler-Schindler\footnote{ETH Zürich, Zürich, Switzerland (klaurin@ehz.ch, asulser@ethz.ch).}  \and Aurelio L. Sulser\footnotemark[1]}
\date{\today}

\mathtoolsset{showonlyrefs}

\begin{document}

\maketitle

\begin{abstract}
  	On a locally finite, infinite tree $T$, let $p_c(T)$ denote the critical probability for Bernoulli percolation. 
  	We prove that every positively associated, finite-range dependent percolation model on $T$ with marginals $p > p_c(T)$ must percolate. Among finite-range dependent models on trees, positive association is thus a favourable property for percolation to occur.

    On general graphs of bounded degree, Liggett, Schonmann and Stacey \cite{Liggett1997} proved that finite-range dependent percolation models with sufficiently large marginals stochastically dominate product measures.
    Under the additional assumption of positive association, we prove that stochastic domination actually holds for arbitrary marginals. Our result thereby generalises Proposition 3.4 in \cite{Liggett1997} which was restricted to the special case $G = \Z$.
   
  	Studying the class of 1-independent percolation models has proven useful in bounding critical probabilities of various percolation models via renormalization. In many cases, the renormalized model is not only 1-independent but also positively associated. This motivates us to introduce the smallest parameter $p_a^+(G)$ such that every positively associated, 1-independent bond percolation model on a graph $G$ with marginals $p > p_a^+(G)$ percolates. We obtain quantitative upper and lower bounds on $p_a^+(\Z^2)$ and on $p_a^+(\Z^n)$ as $n\to \infty$, and also study the case of oriented bond percolation. In proving these results, we revisit several techniques originally developed for Bernoulli percolation, which become applicable thanks to a simple but seemingly new way of combining positive association with finite-range dependence. 
\end{abstract}

\section*{Introduction}

We study finite-range dependent, positively associated bond percolation models. A bond percolation configuration $\omega \in \{0,1\}^E$   on a graph $G=(V,E)$ assigns to each edge $e \in E$ a random status, \emph{open} if $\omega(e)=1$  or \emph{closed} if $\omega(e)=0$, and it is often identified with the random subgraph $(V,\{e \in E: e\ \text{open} \})$. A probability measure $\P$ on the space of configurations $\omega \in \{0,1\}^E$ satisfies 
\begin{itemize}
	\item \emph{positive association} if $\P[\mathcal A_1\cap \mathcal A_2]\ge \P[\mathcal A_1]\cdot \P[\mathcal A_2]$ for all increasing events $\mathcal A_1,\mathcal A_2$,
	\item \emph{k-independence} if the edge statuses in $F_1$ are independent of the edge statuses in $F_2$ whenever $F_1,F_2 \subseteq E$ are disjoint and at distance at least $k$ in $G$.	
\end{itemize}
$\P$ is \emph{finite-range dependent} if it is $k$-independent for some $k \in \mathbb N$. Moreover, $\P$ has \emph{marginals $p \in [0,1]$} if $\P[\omega(e)=1]=p$ for all $e \in E$. The classical and well-studied Bernoulli bond percolation model, where edges are independently declared  open (resp.\ closed) with probability $p$  (resp.\ $1-p$), is positively associated, $0$-independent and has marginals $p$.
We say that a bond percolation model $\P$ on an infinite, connected, locally finite graph \emph{percolates} if the random subgraph $\omega$  contains an infinite connected component with positive probability.   We refer to Section \ref{sec:background-and-notation} for more background and precise definitions.

In this paper, we are interested in the geometry of the random subgraph when we allow for finite-range dependencies and we ask how this large class of models compares with Bernoulli percolation, particularly regarding the existence of an infinite cluster.  
Our first result studies percolation on trees and establishes that among positively associated models, adding finite-range dependencies can only favour the existence of an infinite cluster. 

\begin{theorem}\label{thm:trees}
	Let $T$ be an infinite, locally finite tree. Then every positively associated, finite-range dependent bond percolation model $\P$ with marginals $p>p_c(T) = \frac{1}{\br(T)}$ percolates.
\end{theorem}
The exact identification of the critical probability $p_c(T)$ of Bernoulli percolation in terms of the branching number $\text{br}(T)$ is due to Lyons \cite{Lyons1990}. At this point, we only note that a $d$-ary tree has branching number $d$ and postpone the general definition to Section \ref{sec:background-and-notation}.
We point out that neither the assumption of positive association nor the assumption of finite-range dependence can be dropped. Indeed, without positive association, Balister and Bollobás \cite{Balister2012} show the existence of 1-independent percolation models with marginals arbitrarily close to $3/4$ that do not percolate, and without finite-range dependence, positively associated percolation models with marginals arbitrarily close to $1$ that do not percolate can be obtained by assigning the same state within each level set of some fixed root $o\in T$. However, the assumption of homogeneous marginals $p$ can easily be replaced by a uniform lower bound on $\mathrm{P}[\omega(e)=1]$.

On general graphs of bounded degree, Liggett, Schonmann and Stacey \cite{Liggett1997} proved that every finite-range dependent bond percolation model $\P$ with large enough marginals $p$ stochastically dominates Bernoulli bond percolation with sufficiently small but strictly positive marginals $q$, and that one can take $q$ arbitrarily close to $1$ by taking $p$ sufficiently large. They also discuss that no such result can hold for small marginals $p$, and among others, provide an example of a 1-independent, translation invariant bond percolation model on $\Z$ with marginals $1/2$ which does not stochastically dominate Bernoulli bond percolation for any marginals $q>0$. However, this model is not positively associated. 

For the special case $G = \Z$, it is also proved in \cite{Liggett1997} (see Proposition 3.4 and the discussion on p.~75 of their paper) that even for small marginals $p$, every positively associated, 1-independent bond percolation model $\P$ stochastically dominates Bernoulli bond percolation with sufficiently small but strictly positive marginals. Our second result generalises this result to general graphs of bounded degree. We thereby establish that under the additional assumption of positive association, the main result of Liggett, Schonmann and Stacey extends to arbitrary marginals.

\begin{theorem} \label{thm:stochastic-domination}
	Let $G$ be a graph with maximum degree $\Delta < \infty$ and let $k\ge 1$. There exist increasing homeomorphisms $\rho,\sigma : [0,1] \to [0,1]$ depending only on $\Delta$ and $k$ such that for every positively associated, $k$-independent bond percolation model $\P$ with marginals $p$, we have
	\begin{equation*}
		\pi_{\rho(p)} \ll \P \ll \pi_{\sigma(p)},
	\end{equation*}
	where $\pi_q$ denotes the Bernoulli bond percolation model on $\{0,1\}^E$ with marginals $q$.
\end{theorem}

\begin{remark}
	Theorem \ref{thm:trees} and Theorem \ref{thm:stochastic-domination} are not specific to bond percolation, and it will become clear in the respective proofs that they also apply to site percolation.
\end{remark}

In the following, we restrict our attention to 1-independent percolation models which have proven useful in bounding critical probabilities of various percolation models via renormalization (see \cite{Balister2005} and Section 6.2 in \cite{bollobas2006percolation}). Balister and Bollobás \cite{Balister2012} systematically introduced the class of 1-independent bond percolation models on a connected, infinite, locally finite graph $G$ with marginals at least $p$ (i.e.\ $\P[\omega(e)=1]\ge p$ for all $e \in E$), denoted here by $\mathcal{P}(G,p)$, and defined 
\begin{equation}
	p^{+}(G) := \sup\{p : \exists\; \P \in \mathcal{P}(G,p) \text{ that does not percolate}\}. 
\end{equation}
In other words, for  $p > p^{+}(G)$, any model in $\mathcal{P}(G,p)$ percolates. In contrast, for $p < p^{+}(G)$, there exists a model in $\mathcal{P}(G,p)$ that does not percolate. Known bounds on the quantity $p^{+}(G)$ for different graphs $G$ will be discussed shortly. 

Inspired by this systematic approach and motivated by the fact that standard renormalization schemes often preserves positive association (see, e.g., Section 6.2 in \cite{bollobas2006percolation}), we introduce the subclass $\mathcal{P}_{a}(G,p)$ of positively associated models in $\mathcal{P}(G,p)$, and define
\begin{equation}
	p_a^{+}(G) := \sup\{p : \exists\; \P \in \mathcal{P}_a(G,p) \text{ that does not percolate}\}. 
\end{equation}
Note that restricting $\mathcal{P}_{a}(G,p)$  to models with all marginals equal to $p$ would not affect the quantity $p_a^{+}(G)$ since a non-percolating model $\P' \in \mathcal{P}_a(G,p)$ with marginals equal to $p$ can be obtained from a non-percolating model $\P \in \mathcal{P}_a(G,p)$  by independently deleting each edge $e$ with probability $1 - p/\P[\omega(e)=1]$.
Defining $p_c(G)$ as usual to be the supremum over all marginals $p$ for which Bernoulli bond percolation does not percolate, we clearly have 
\begin{equation}
    p^+(G) \ge  p_a^+(G) \ge p_c(G).
\end{equation} 

For infinite, locally finite trees, it was proven in \cite[Thm.\ 1.2--1.3]{Balister2012} that 
\begin{equation*}
	p^+(T) = \begin{cases}
		1- \frac{\br(T)-1}{\br(T)^2} & \text{if } \br(T)< 2, \\
		\frac{3}{4} & \text{if } \br(T) \ge 2. 
	\end{cases}
\end{equation*}
As an immediate corollary of Theorem \ref{thm:trees}, we obtain the equality 
	\begin{equation*}
		p_a^+(T) = p_c(T) =  \frac{1}{\textrm{br}(T)}.
	\end{equation*}
Particularly for $\br(T)$ large, the behaviour of $p^+(T)$ and $p^+_a(T)$ is very different: While $p_a^+(T) \to 0$ as $\textrm{br}(T) \to \infty$, $p^+(T)$ is uniformly bounded from below by $ 3/4$.

One might ask if positive association also makes a difference for the lowest marginal at which percolating models exist. Following \cite{Balister2012}, define the largest parameter $p^-(G)$ such that no 1-independent bond percolation model on $G$ with marginals $p < p^-(G)$ percolates. Similarly, we define $p_a^-(G)$ by additionally restricting to positively associated models. Clearly, we have $p^-(G) \le  p_a^-(G) \le p_c(G)$, and a simple argument  (see \cite[Prop.~4.1]{Balister2012}) implies
\begin{equation}
    p^-(G) \le p_a^-(G) \le p_c^{\mathrm{site}}(G)^2.
\end{equation}
In the case of trees, the matching lower bound on  $p^-(T)$ is determined in \cite[Thm.\ 4.2]{Balister2012}, and so we deduce
\begin{equation*}
	p^-(T) = p_a^-(T) = \frac{1}{\br(T)^2},
\end{equation*}
showing that the additional assumption of positive association is not consequential in this case. 

Unlike for trees, where all six parameters, $p_c(T)$, $p^{\mathrm{site}}_c(T)$, $p^+(T)$, $p_a^+(T)$, $p^-(T)$, and $p_a^-(T)$, can be exactly identified, this is typically not possible for other graphs, even for $p_c(G)$, and instead, one tries to obtain good bounds. In the following, we focus on the study of $p^{+}(G)$ and  $p_a^+(G)$ on the lattice $\Z^n$ and on the directed lattice\footnote{The directed lattice $\vv{\Z}^n$ has vertex set $\Z^n$ and edge set $\vv{E}(\Z^n)$ consisting of all directed edges $(x,y)$ satisfying $\{x,y\} \in E(\Z^n)$ and $x_i \le y_i$ for every $1\le i\le n$. We note that the definition of $p^{+}(G)$ and $p_a^{+}(G)$ applies analogously to directed graphs.} $\vv{\Z}^n$. 

For oriented Bernoulli bond percolation on $\vv{\Z}^n$, it was proven in \cite{Cox1983} that 
\begin{equation}
    p_c(\vv{\Z}^n) = \frac{1}{n} \cdot (1+o(1)) \quad \text{as} \ n \to \infty. 
\end{equation} 
The next theorem exactly determines $p_a^+(\vv{\Z}^n)$ up to the first-order term. 
\begin{theorem} \label{thm:oriented-d-large}
	\begin{equation*}
		p_a^+(\vv{\Z}^n) = \frac{\sqrt{2}}{n} \cdot (1+o(1)) \quad \text{as}\ n \to \infty.
	\end{equation*}
\end{theorem}
While $p_a^+(\vv{\Z}^n)$ also goes to 0 as $n\to \infty$, we note that the first-order term differs. This shows that there exist positively associated, 1-independent models which `percolate worse' than oriented Bernoulli bond percolation even as $n \to \infty$. 

In the unoriented case, Kesten \cite{Kesten1990} and Gordon \cite{Gordon1991} studied the asymptotic behaviour of $p_c(\Z^n)$ and proved $p_c(\Z^n) = \frac{1}{2n} \cdot (1+o(1))$ as $n \to \infty$ (see also \cite{Hara1990,Bollobas1994} for refined asymptotics). It is interesting to compare this with the best known lower bound on $\lim_{n \to \infty} p^+(\Z^n)$: In \cite{Day2020}, it is proven that for all $n\ge 2$,
\begin{equation*}
	p^+(\Z^n) \ge 4 - 2 \sqrt{3} \approx 0.536.
\end{equation*}
The following direct corollary of Theorem \ref{thm:oriented-d-large} shows a very different behaviour of $p_a^+(\Z^n)$.
\begin{corollary}\label{thm:unoriented-d-large}
    \begin{equation*}
		p_a^+(\Z^n) \le   \frac{\sqrt{2}}{n} \cdot (1+o(1)) \quad \text{as}\ n \to \infty.
	\end{equation*}
\end{corollary}
 This upper bound on $p_a^+(\Z^n)$ differs by a factor of $2\sqrt{2}$ from the natural lower bound provided by Bernoulli bond percolation. We would like to point out that it should be relatively straightforward to obtain the upper bound $p_a^+(\Z^n) \le \frac{1}{\sqrt{2}n} \cdot (1+o(1))$ by adapting our proof of Theorem \ref{thm:oriented-d-large} to the unoriented case (see Section \ref{sec:high-dimensional-percolation} for further comments). As we are lacking a non-trivial lower bound, we have no strong reason to believe that this improved upper bound is in fact optimal, and therefore, we decided not to provide a separate proof in the unoriented case.

1-independent bond percolation models were initially studied on the square lattice $\Z^2$ in \cite{Balister2005}. The best known bounds are currently due to \cite{Balister2022} and given by
\begin{equation*}
	0.55 \approx \frac{35-3\sqrt{33}}{32}\le p^+(\Z^2) \le 0.8457.
\end{equation*} 
For the subclass of positively associated models, we prove the existence of a positively associated bond percolation model on $\Z^2$ with marginals strictly larger than $1/2$ that does not percolate and establish a quantitative upper bound on $p^+_a(\Z^2)$.
\begin{theorem}\label{thm:unoriented-d-2}
	\begin{equation*}
	\frac{1}{2} < p^+_a(\Z^2) \le 0.77.
	\end{equation*} 
\end{theorem}
On the one hand, restricting to positively associated models yields better upper bounds. On the other hand, the models constructed to bound $p^+(\Z^2)$ from below are not positively associated and so, even the qualitative lower bound $p^+_a(\Z^2) > \frac{1}{2}$ requires a new construction.
As in the high-dimensional case, we also study oriented percolation in two dimensions.
\begin{theorem}\label{thm:oriented-d-2}
	\begin{equation*}
		p_c^{\mathrm{site}}(\vv{\Z}^2) \le p_a^+(\vv{\Z}^2) \le  \sqrt{p_c^{\mathrm{site}}(\vv{\Z}^2)}.
	\end{equation*} 
\end{theorem}
Combined with the lower bound $p_c^{\mathrm{site}}(\vv{\Z}^2) \ge 2/3$ \cite{Bishir1963} and the upper bound $p_c^{\mathrm{site}}(\vv{\Z}^2) \le 0.7491$ \cite{balister1994improved}\footnote{As pointed out in \cite{balister1994improved}, the upper bound $p_c^{\mathrm{site}}(\vv{\Z}^2) \le 0.72599$ given in \cite{Balister1993} is not correct. It should have been $p_c^{\mathrm{site}}(\vv{\Z}^2) \le 0.762$.}, the theorem yields the quantitative bounds $2/3 \le p^+_a(\vv{\Z}^2) \le  0.8656.$

\subsection*{Organization of the Paper}
In the next section, we introduce necessary background and notation. In Section \ref{sec:trees}, we present the proof of Theorem \ref{thm:trees}. In Section \ref{sec:stochastic-domination}, we prove a simple but crucial lemma on combining positive association and finite-range dependence, which we then use to establish Theorem \ref{thm:stochastic-domination} as well as a stochastic domination result that will allow us to restrict our attention `levelwise independent' models. In Section \ref{sec:two-dimensional-percolation}, we present the proofs of Theorems \ref{thm:unoriented-d-2} and \ref{thm:oriented-d-2} on two-dimensional percolation. High-dimensional percolation is studied in Section \ref{sec:high-dimensional-percolation}, where we prove Theorem \ref{thm:oriented-d-large}. Finally, in Section \ref{sec:open-questions}, we discuss some directions for future research.

\subsection*{Acknowledgements}
We are grateful to Vincent Tassion for proposing the study of positively associated, finite-range dependent percolation models, and we  would like to thank him for many stimulating discussions and insightful inputs, particularly regarding Theorem \ref{thm:trees}. We also thank François Bienvenu for valuable inputs related to branching processes.

The first author is part of NCCR SwissMAP and  has received funding from the European Research Council (ERC) under the European Union’s Horizon 2020 research and innovation program (grant agreement No 851565).

\section{Background and Notation}
\label{sec:background-and-notation}

\paragraph{Graph Distances.} Let $G$ be an undirected graph with vertex set $V$ and edge set $E$. An undirected path $\gamma=(\gamma_i)_{i=0}^n \subseteq V$ is a sequence of distinct vertices such that for every $1\le i \le n$, $\{\gamma_{i-1},\gamma_i\} \in E$. For vertices $u,v \in V$, edges $e,f \in E$, the usual graph distances are given by 
\begin{align*}
	&d(u,v):=\inf\{n \ge 0 : \exists \ \text{path}\ \gamma \ \text{with} \ \gamma_0 = u, \gamma_n =v \} \in \mathbb N \cup \{\infty\}, \\
	&d(e,v):=\inf\{d(u,v) : u \in e\},  \\
	& d(e,f):=\inf\{d(u,v) : u \in e, v \in f \}.  
\end{align*}
Theses distances naturally extend to subsets $U \subseteq V$ and $F \subseteq E$ by setting, for example, $d(e,F):= \inf\{d(e,f):f \in F\}$ and $d(e,U):= \inf\{d(e,u):u \in U\}$.
Moreover, for $F \subseteq E$ and $k\ge 0$, we define the \emph{$k$-closure} of $F$ by
\begin{equation}
	\cl_k(F) = \begin{cases}
		\{e \in E : d(e,F) < k\} &\text{if}\ k \ge 1,\\
		F &\text{if}\ k=0.
	\end{cases}
\end{equation}
and the \emph{$k$-edge boundary} $\Delta_k F := \cl_k(F) \setminus F$.
When considering graph distances on a directed graph $G=(V,\vv{E})$, we will, unless mentioned explicitly, refer to the graph distances in the corresponding undirected graph, i.e.\ with respect to the length of the shortest undirected path.

In an undirected graph, we define the neighborhood of a vertex set $W\subseteq V$ as $N(W)=\{v \in V: \{u,v\} \in E\ \text{for some}\ u \in W\}$. The degree of a vertex $v \in V$ is then defined as the size of its neighborhood $N(v):= N(\{v\})$.  In a directed graph, we define two neighborhoods of a vertex set $W$, namely $N_+(W)=\{v \in V: (u,v) \in \vv{E}\ \text{for some}\ u \in W\}$ and  $N_-(W)=\{u \in V: (u,v) \in \vv{E}\ \text{for some}\ v \in W\}$. The outdegree (resp.\ indegree) of a vertex $v \in V$ is then defined as the size of $N_+(W)$ (resp.\ $N_-(W)$).

\paragraph{Percolation Model.} Given an undirected (resp.\ directed) graph $G$ with vertex set $V$ and edge set $E$ (resp.\ $\vv{E}$), we refer to a probabilitiy measure $\mathrm{P}$ on the space of bond percolation configurations $\{0,1\}^E$ (resp. $\{0,1\}^{\vv{E}}$) equipped with the product sigma-algebra as a bond percolation model (resp.\ oriented bond percolation model). A probability measure $\mathrm{P}$ on the space of site percolation configurations $\{0,1\}^V$ equipped with the product sigma-algebra is called a site percolation model. 

For $F \subseteq E$ (resp. $\vv{E}$), we say that an event is \emph{supported on $F$} if it is measurable with respect to the status of the edges in $F$. 
We write $\{u \xleftrightarrow{\omega} v\}$ for the event that there exists an open path from $u$ to $v$ in the percolation configuration $\omega$. In the case of a directed graph, we similarly write $\{u \xrightarrow{\omega} v\}$ if there exists an open directed path from $u$ to $v$.
We denote the cluster of a vertex $v \in V$ by $\mathcal{C}_v(\omega) := \{u \in V : v \xleftrightarrow{\omega} u\}$. In the case of an directed graph, we similarly write $\vv{\mathcal{C}}_v(\omega) := \{u \in V : v \xrightarrow{\omega} u\}$ for the directed cluster of $v$ that contains all vertices that are reachable from $v$. 

We say that a percolation model $\mathrm{P}$ on an undirected (resp.\ directed) graph \emph{percolates} if the percolation configuration $\omega$ sampled according to $\mathrm{P}$ contains an infinite cluster (resp.\ an infinite directed cluster) with positive probability. We write $\pi_p$ for the Bernoulli bond percolation model on an undirected or directed graph $G$ and write $p_c(G):= \sup\{p : \pi_p\ \text{does not percolate}\}$ for the critical probability. Analogously, we write $\pi_p^{\mathrm{site}}$ for the Bernoulli site percolation model on $G$ with critical probability $p_c^{\mathrm{site}}(G)$.

\paragraph{Positive Association.} A function $f:\{0,1\}^E \to \mathbb R$ is \emph{increasing} if $f(x) \le f(y)$ whenever $x \le y$ (in the natural partial ordering on $\{0,1\}^E$) and an event $A$ is increasing if $\mathbf{1}_A$ is increasing. A probability measure $\mathrm{P}$ is positively associated if any two increasing events $A, B$ are positively correlated under $\mathrm{P}$, i.e.
\begin{equation*}
	\mathrm{P}\left[A \cap B\right] \ge \mathrm{P}\left[A \right] \cdot \mathrm{P}\left[B\right].
\end{equation*}
This is usually referred to as the FKG inequality \cite{fortuin1971correlation}. 

\paragraph{Finite-range Dependence.} 
For $k \in \Z_+ := \{x \in \Z: x \ge 0\}$, we say that a bond percolation model $\mathrm{P}$  is $k$-independent if if the edge statuses in $F_1$ are independent of the edge statuses in $F_2$ whenever $F_1,F_2 \subseteq E$ are disjoint and at distance at least $k$ in $G$. Moreover, $\mathrm{P}$ is called finite-range dependent if it is $k$-independent for some $k \in \Z_+$. 

\paragraph{Stochastic Domination.} Let $S$ be a Polish space equipped with the Borel $\sigma$-algebra $\mathcal S$ and a closed partial ordering $\leq$. As before, a function $f:S \to \mathbb R$ is \emph{increasing} if $f(x) \le f(y)$ for every $x \leq y$.  A probability measure $\mu$ is said to be \emph{stochastically dominated} by another probability measure $\nu$, denoted $\mu \preceq \nu$, if for every increasing and bounded measurable function $f:S \to \mathbb R$, 
\begin{equation}
	\mathbb{E}_\mu[f] \le \mathbb{E}_\nu[f].
\end{equation}
Strassen's theorem \cite{Strassen1965} (see also \cite{Lindvall1999}) states that two probability measures $\mu$ and $\nu$ satisfy $\mu \preceq \nu$ if and only if there exists a coupling of $X \sim \mu$ and $Y \sim \nu$ (on some probability space) such that almost surely $X \le Y$. 

We will make use of the following result due to Kamae, Krengel and O’Brien
for the comparison of two stochastic processes that is based on Strassen’s theorem. Let $(S_n)_{n\ge 0}$ be Polish spaces, each equipped with the Borel $\sigma$-algebra and a closed partial ordering $\leq$. Write $S^n$ for $S_0 \times \ldots \times S_n$, $x^n = (x_0,\ldots,x_n)$ for an element of $S^n$, and $x^n \le y^n$ iff $x_i \le y_i$ for every $0\le i\le n$.
\begin{theorem}[Theorem 2 in \cite{Kamae1977}] \label{thm:Kamae}
	Let $\mathrm{P}_0$ and $\mathrm{Q}_0$ be probability measures on $S_0$ satisfying $\mathrm{P}_0 \preceq \mathrm{Q}_0$. For $n\ge 1$, let $p_n,q_n$ be stochastic kernels on $S^{n-1}\times S_n$ satisfying
	\begin{equation}
		\forall x^{n-1} \le y^{n-1}, \quad p_n(x^{n-1},\boldsymbol{\cdot}) \preceq q_n(y^{n-1},\boldsymbol{\cdot}).
	\end{equation}
	Then there exist stochastic processes $(X_n)_{n\ge 0}$ and $(Y_n)_{n\ge 0}$ (on some probability space) such that $X_0 \sim \mathrm{P}_0$, $Y_0 \sim \mathrm{Q}_0$, for every $n\ge 1$ and $x^{n-1}$ (resp.\ $y^{n-1}$), the conditional distribution of $X_n$ given $X^{n-1} = x^{n-1}$ (resp.\ $Y_n$ given $Y^{n-1} = y^{n-1}$) is $p_n(x^{n-1},\boldsymbol{\cdot})$ (resp.\ $q_n(y^{n-1},\boldsymbol{\cdot})$), and almost surely
	\begin{equation}
		\forall n \ge 0, \quad X_n \le Y_n.
	\end{equation}
\end{theorem}
In Section \ref{sec:two-dimensional-percolation}, we will apply also this result in the Markovian setting where the transition kernel depends on $x^{n-1}$ only via the previous state $x_{n-1}$ and $(S_n)_{n\ge 0}$ are countable. In this case, it is convenient to view the process as a time-homogeneous Markov chain on the larger state space $S := \bigcup_{n\ge 0} S_n$, with initial distribution supported on $S_0$ and with transition probability $p$ supported on transitions from $S_{n-1}$ to $S_n$. The condition on the stochastic kernels is then replaced by the condition that the transition probabilities $p$, $q$ satisfy
\begin{equation}
	\forall n \ge 0,\ \forall x \le y \in S_n, \quad p(x,\boldsymbol{\cdot}) \preceq q(y,\boldsymbol{\cdot}).
\end{equation}

\paragraph{Branching Number.} To identify the percolation threshold for infinite, locally finite trees, the branching number was first introduced in \cite{Lyons1990}. For a detailed discussion, we refer to \cite{Lyons2016}. We will make use of the definition presented in Proposition 2.1 in \cite{Lyons1990}. A flow on an infinite, locally finite tree $T=(V,E)$ with root $o \in V$ is a non-negative function $\theta$ on $V$ such that for all $v \in V$,
$$
\theta(v)=\sum_{w \in N^+(v)} \theta(w) .
$$
A flow $\theta$ such that $\theta(o)=1$ is called a unit flow. The set of unit flows on $T$ is denoted $U(T)$. We define the branching number as 
\[\operatorname{br} T=\sup _{\theta \in U(T)} \liminf_{d(o,v) \rightarrow \infty} \theta(v)^{-1 /d(o,v)} .\]
While the definition a priori depends on the choice of the root, one finds that it actually does not depend on the root.

\section{Proof of Theorem \ref{thm:trees}}
\label{sec:trees}

Let $T=(V,E)$ be an infinite, locally finite tree and fix some root $o \in V$. Let $\P$ be a positively associated, $k$-independent bond percolation model with marginals $p>1/\br(T)$ and sample $\omega\sim\P$ on a general probability space whose measure we denote by $\mathbb P$. Our goal is to prove that
\begin{equation}\label{eq:7}
	\mathbb P[o\xleftrightarrow{\omega} \infty] >0.
\end{equation}
The second moment method is the standard tool to establish the existence of an infinite cluster on a tree, and we also follow this route. However, we first make use of an inhomogeneous, independent depletion of edges to facilitate the computation of the first and second moment. This key idea is due to Vincent Tassion who has kindly given permission to publish it in this form.

To this end, we write $q_x := \P[o \xleftrightarrow{\omega}x]$ for the probability that  $x \in V$ is connected to the root, which satisfies $q_x \ge p^{d(o,x)}$ due to positive association, and define for every edge $e=\{x,y\} \in E$ with $x$ closer to the root than $y$, 
\begin{equation}
	p_e := \frac{q_x}{q_y} \cdot p.
\end{equation}
Importantly, positive association implies that $q_y \ge q_x \cdot p$ and so $p_e \in [0,1]$ for every  $e \in E$. 
Now, let $\xi = (\xi_e)_{e \in E } \in \{0,1\}^E$ be a family of random variables with $\mathbb P[\xi_e = 1] = p_e$ that are independent of each other and of $\omega$. We define the bond percolation $\eta = (\eta_e)_{e\in E}$ by 
\begin{equation}
	\eta_e = \omega_e \cdot \xi_e.
\end{equation}
In words, it is obtained from $\omega$ by independently depleting each $\omega$-open edge $e$ with probability $1-p_e$. 
We note that the law of $\eta$ is $k$-independent and positively associated (as an increasing function of $\omega$ and $\xi$). Introducing $\eta$ is motivated by the observations that for all $x \in V$ and for all $x' \in [o,x]$, the unique path from $o$ to $x$, we have
\begin{align} \label{eq:9}
	 &\mathbb P[o \xleftrightarrow{\eta} x] = q_x \cdot \prod_{e \in [o,x]} p_e = p^{d(o,x)}, \quad \text{and}\\ 
	\label{eq:10}
  &\mathbb P[x' \xleftrightarrow{\eta} x] \le \frac{\mathbb P[o \xleftrightarrow{\eta} x]}{\mathbb P[o \xleftrightarrow{\eta} x']} = p^{d(x',x)}.
\end{align}

We are now ready to prove 
\begin{equation}\label{eq:8}
	\mathbb P[o\xleftrightarrow{\eta} \infty] >0,
\end{equation}
which then implies \eqref{eq:7} since $\omega \ge \eta$. From now on, we closely follow the proof of Theorem 6.2 in \cite{Lyons1990}. Denote by $V_n$ the vertices at distance $n$ from the root. For any unit flow $\theta$, set
\begin{equation}
	X_n^\theta(\eta) := \sum_{x \in V_n} \frac{\theta(x)}{p^{n}} \mathbf{1}_{\{o \xleftrightarrow{\eta}x\}}.
\end{equation}
We note that by \eqref{eq:9}, the first moment of $X_n^\theta$ is simply 
\begin{equation}
	\mathbb E[X_n^\theta] = \sum_{x \in V_n} \frac{\theta(x)}{p^{n}} \mathbb P[o \xleftrightarrow{\eta}x] = \sum_{x \in V_n} \theta(x) = \theta(o) = 1.
\end{equation}
We are left with bounding the second moment. To this end, for $x,y \in V$, we denote by $z = x\wedge y$ the last vertex on the path from $o$ to $x$ which is also on the path from $o$ to $y$, and by $z'$ the vertex on the path from $z$ to $y$ which is at distance exactly $k$ from $z$ (by convention, $z' = y$ if $d(z,y)<k$). Then we have 
\begin{align}
	\mathbb P[o \xleftrightarrow{\eta} x, o \xleftrightarrow{\eta} y] \le \mathbb P[o \xleftrightarrow{\eta} x] \cdot \mathbb P[z' \xleftrightarrow{\eta} y] \le p^{d(o,x)} \cdot p^{d(z,y)-k},
\end{align}
where we used the $k$-independence in the first  and \eqref{eq:9} together with \eqref{eq:10} in the second inequality. Hence,
\begin{align}
	\mathbb E\Big[\big(X_n^\theta\big)^2\Big] &= \sum_{x,y \in V_n} \frac{\theta(x)\theta(y)}{p^{2n}} \mathbb P[o \xleftrightarrow{\eta} x, o \xleftrightarrow{\eta} y] \\
	&\le \sum_{z:\, d(o,z)\le n}\ \sum_{x,y \in V_n:\, x\wedge y = z} \frac{\theta(x)\theta(y)}{p^{d(o,z)+k}} \\
	&= p^{-k} \sum_{z:\, d(o,z)\le n} p^{-d(o,z)} \sum_{x,y \in V_n:\, x\wedge y = z} \theta(x)\theta(y) \\
	&\le p^{-k} \sum_{z:\, d(o,z)\le n} p^{-d(o,z)} \theta(z)^2
\end{align}
Since $p^{-1} < \br(T)$, we may choose $\theta$ such that $C:=\sum_{z \in V} p^{-d(o,z)} \theta(z)^2 <\infty$, which implies that $\sup_n 	\mathbb E[(X_n^\theta)^2]$ is finite. Using the Paley-Zygmund inequality, we deduce
\begin{equation}
	\lim_{n \to \infty} \mathbb P[ X_n^\theta >0] \ge \lim_{n \to \infty} \frac{\mathbb E\big[X_n^\theta\big]^2}{\mathbb{E}\Big[\big(X_n^\theta\big)^2\Big]} \ge \frac{p^k}{C},
\end{equation} 
which implies \eqref{eq:8} and thereby completes the proof.

\section{Positive Association and Finite-range Dependence}
\label{sec:stochastic-domination}

We begin with a simple but powerful lemma combining positive association and $k$-independence. It provides a way of controlling the propagation of information in positively associated, $k$-independent percolation models, and will thus be a crucial ingredient in proving Theorems \ref{thm:stochastic-domination}--\ref{thm:oriented-d-2}. 

Let us introduce the following notation. For $F \subseteq E$, we say that an event $A$ is \emph{$F$-increasing} if 
\begin{equation*}
	\left( \omega \in A, \ \omega \le \omega' \ \text{on} \ F, \ \text{and} \ \omega = \omega' \ \text{on} \ F^c \right) \implies \omega' \in A.
\end{equation*}
Intuitively, an event is $F$-increasing if it is favoured by the occurrence of open edges in $F$. However, open edges outside of $F$ might not be favourable. By definition, an $F$-increasing event is $F'$-increasing for every subset $F' \subseteq F$. In the special case $F=E$, we recover the classical notion of an increasing event.
We analogously say that $A$ is \emph{$F$-decreasing} if 
\begin{equation*}
	\left( \omega \in A, \ \omega \ge \omega' \ \text{on} \ F, \ \text{and} \ \omega = \omega' \ \text{on} \ F^c \right) \implies \omega' \in A,
\end{equation*}
and note that $A$ is $F$-decreasing if and only if $A^c$ is $F$-increasing.

\begin{lemma}\label{lem:positive-association-k-independence}
	Let $\P$ be a probability measure on  $\{0,1\}^E$ and $k\ge 0$. Then the following are equivalent:
	\begin{itemize}
		\item[(i)] $\P$ satisfies positive association and $k$-independence.
		\item[(ii)]  For every $F \subseteq E$, for every $F$-increasing event $A$ supported on $F$, and for every $\cl_k(F)$-increasing event $B$,
		\begin{equation}
			\P[A\cap B] \ge \P[A] \cdot \P[B].
		\end{equation}
	\end{itemize}
\end{lemma}
\begin{proof}
We prove the equivalence for finite $E$. Once this is established, the infinite case follows by taking suitable limits. 

\noindent  $(i) \implies (ii)$: Fix $F \subseteq E$, a $F$-increasing event $A$ supported on $F$, and a $\cl_k(F)$-increasing event $B$. 
We make use of the observation that as a $\cl_k(F)$-increasing event, $B$ can be partitioned as
\begin{equation}
	B = \bigsqcup_{\eta \in \{0,1\}^{(\cl_k(F))^c}} \left(  B_\eta \cap \{\omega : \omega_e = \eta_e, \forall e \in (\cl_k(F))^c\} \right),
\end{equation} 
where $B_\eta$ are taken to be $\cl_k(F)$-increasing events  supported on $\cl_k(F)$, and thus $E$-increasing. Let us further partition $B$ as
\begin{equation}
	B =  \bigsqcup_{\eta \in \{0,1\}^{(\cl_k(F))^c}} \left(  B_\eta \cap C_\eta \cap D_\eta\right),
\end{equation}
where 
\begin{align}
	&C_\eta := \{\omega: w_e = 1, \forall e \in (\cl_k(F))^c\ \text{with}\ \eta_e = 1 \}\ \text{is $E$-increasing,} \\
	\text{and}\ &D_\eta := \{\omega: w_e = 0, \forall e \in (\cl_k(F))^c\ \text{with}\ \eta_e = 0 \}\ \text{is $E$-decreasing}.
\end{align}
We make the important observation that the event $D_\eta$ is supported on the complement of $\cl_k(F)$ while $A$ is supported on $F$. Therefore, the events $D_\eta$ and $A$ are independent due to the $k$-independence of $\P$. Thus, for every $\eta \in \{0,1\}^{(\cl_k(F))^c}$,
\begin{align}
	\P[A \cap B_\eta \cap C_\eta \cap D_\eta] &= \underbrace{\P[A \cap D_\eta]}_{=\P[A] \cdot \P[D_\eta]} - \underbrace{\P[A \cap (B_\eta \cap C_\eta)^c \cap D_\eta ]}_{\le \P[A] \cdot \P[(B_\eta \cap C_\eta)^c \cap D_\eta ]} \\
	&\ge \P[A] \cdot \left(\P[D_\eta] - \P[(B_\eta \cap C_\eta)^c \cap D_\eta ]\right) \\
	&= \P[A] \cdot \P[B_\eta \cap C_\eta \cap D_\eta ], 
\end{align}
where the inequality follows from the positive association of $\P$ since the event $A$ is $E$-increasing and the event $(B_\eta \cap C_\eta)^c \cap D_\eta$ is $E$-decreasing. Summing over all $\eta \in \{0,1\}^{(\cl_k(F))^c}$ implies $\P[A \cap B] \ge \P[A]\cdot \P[B]$.

We now prove $(ii) \implies (i)$. First, note that positive association directly follows from $(ii)$ by setting $F=E$ since $\cl_k(E)=E$. Second, let $A$ be an $F$-increasing event supported on some $F \subseteq E$ and let $B$ be an event supported on $(\cl_k(F))^c$. Since $B$ is both $\cl_k(F)$-increasing and $\cl_k(F)$-decreasing, $(ii)$ implies 
\begin{equation}
	\P[A \cap B] \ge \P[A]\cdot \P[B] \quad \text{and} \quad 	\P[A \cap B] \le \P[A]\cdot \P[B],
\end{equation}  
thus $A$ and $B$ are independent. It follows analogously that $A$ and $B$ are independent if $A$ is $F$-decreasing instead of $F$-increasing. This concludes the proof since a general event $A$ supported on $F$ can be partitioned into intersections of increasing and decreasing events. 
\end{proof}

We now give a simple example to illustrate how we will make use of this lemma. 
\begin{example}
	Consider a cycle of length $4$, i.e.\ the graph with vertex set $\{1,2,3,4\}$ and edges $\{1,2\}$, $\{2,3\}$, $\{3,4\}$, and $\{4,1\}$. 
    For every positively associated, 1-independent bond percolation model $\P$, Lemma \ref{lem:positive-association-k-independence} implies 
	\begin{equation}
		\P[\omega(\{1,2\})=1 \mid \omega(\{4,1\})=1, \omega(\{2,3\})=1,\omega(\{3,4\})=0] \ge  	\P[\omega(\{1,2\})=1]
	\end{equation}
    since we condition on a $\cl_1(\{1,2\})$-increasing event.
	Informally speaking, the negative information about the edge $\{3,4\}$ cannot propagate through the \emph{open} 1-edge boundary. 
 
    To illustrate the important role of positive association, let $X_1,\ldots,X_4$ be independent, Ber($1/2$)-distributed and define the 1-independent bond percolation model $\P^\ast$ by setting $\omega(\{i,j\})= \mathbf{1}_{X_i=X_j}$. It is easy to see that 
	\begin{equation}
		\P^\ast[\omega(\{1,2\})=1 ] = 1/2 = \P^\ast[\omega(\{1,2\})=1 \mid \omega(\{4,1\})=1, \omega(\{2,3\})=1],
	\end{equation}
	but 
	\begin{equation}
		\P^\ast[\omega(\{1,2\})=1 \mid \omega(\{4,1\})=1, \omega(\{2,3\})=1,\omega(\{3,4\})=0] = 0.
	\end{equation}
	Thus, even though the edge $\{1,2\}$ is open with conditional probability $1/2$ given its open 1-edge boundary, the negative information about the edge $\{3,4\}$ propagates through the \emph{open} 1-edge boundary and forces the edge $\{1,2\}$ to be closed. 
\end{example}

Having Lemma \ref{lem:positive-association-k-independence} at our disposal, we are now ready to prove Theorem \ref{thm:stochastic-domination} by extending the proof of Proposition 3.4 in \cite{Liggett1997} to graphs of bounded degree.

\begin{proof}[Proof of Theorem \ref{thm:stochastic-domination}]
	Let $G$ be a graph with maximum degree $\Delta < \infty$ and $\P$ be a positively associated, $k$-independent bond percolation model with marginals $p \in [0,1]$. We aim to construct an increasing homeomorphism $\rho:[0,1]\to[0,1]$ depending only on $k$ and $\Delta$ such that 
	\begin{equation}\label{eq:1}
		\pi_{\rho(p)} \ll \P. 
	\end{equation}
	This will then also imply $\P \ll \pi_{\sigma(p)}$ with $\sigma(p)=1-\rho(1-p)$. Indeed, by switching the roles of $0$ and $1$, \eqref{eq:1} implies that every positively associated, $k$-independent bond percolation model with marginals $1-q$ is dominated by Bernoulli bond percolation with marginals $1-\rho(q)$. 
	
	To establish \eqref{eq:1}, we follow the proof of Proposition 3.4 in \cite{Liggett1997} closely. By fixing an arbitrary order, we may assume $E=\mathbb N$ for the rest of the proof. Consider a family of random variables $(X_n)_{n \in \mathbb N}$ with  law $\P$ on a general probability space whose measure we denote by $\mathbb P$. 
	Moreover, let $(Y_n)_{n\in \mathbb N}$ be a family of random variables, independent of $(X_n)_{n \in \mathbb N}$, with  law $\pi_{p'}$ for some $p'$ specified below. We define $(Z_n)_{n \in \mathbb N}$ by $Z_n = X_n \cdot Y_n$. 
	It now suffices to prove that for any finite subset $I = \{n_1,\ldots,n_{j+1}\} \subseteq \mathbb N$ with $n_1 < \ldots < n_j < n_{j+1}$ and any choice $\varepsilon_1,\ldots,\varepsilon_j \in \{0,1\}$, we have
	\begin{equation}\label{eq:2}
		\mathbb P[Z_{n_{j+1}} = 1 \mid Z_{n_1} = \varepsilon_1,\ldots,Z_{n_j}=\varepsilon_j] \ge \rho(p)
	\end{equation}
	whenever $\mathbb P[Z_{e_{n_1}} = \varepsilon_1,\ldots,Z_{e_{n_j}}=\varepsilon_j]>0$. According to Lemma 1.1 in \cite{Liggett1997}, \eqref{eq:2} then implies that the law of $(Z_n)_{n \in \mathbb N}$ stochastically dominates $\pi_{\rho(p)}$, and thus, \eqref{eq:1} follows. 
	
	For fixed $I$,  let $K\subseteq I$ be the set of edges in $I$ who are at distance strictly less than $k$ from the edge $n_{j+1}$. Note that $K$ has size at most $2(\Delta-1)^k$. For fixed $\varepsilon_1,\ldots,\varepsilon_j \in \{0,1\}$, denote by $K_0$ the subset of $K$ for which $\varepsilon_i = 0$. First, observe that 
	\begin{equation}
		\mathbb P[Z_{n_{j+1}} = 0 \mid Z_{n_1} = \varepsilon_1,\ldots,Z_{n_j}=\varepsilon_j] = p' \cdot \mathbb P[X_{n_{j+1}} = 0 \mid Z_{n_1} = \varepsilon_1,\ldots,Z_{n_j}=\varepsilon_j] + (1 - p'),
	\end{equation}
	and so \eqref{eq:2} follows once we have shown that 
	\begin{equation}\label{eq:3}
		\mathbb P[X_{n_{j+1}} = 0 \mid Z_{n_1} = \varepsilon_1,\ldots,Z_{n_j}=\varepsilon_j] \le \frac{p' - \rho(p)}{p'}.
	\end{equation}
	To this end, we upper bound the conditional probability in \eqref{eq:3} by 
	\begin{equation}\label{eq:4}
		\frac{\mathbb P[X_{n_{j+1}} = 0;\ Z_{n_i} = \varepsilon_i, \forall n_i \in I \setminus K_0]}{\mathbb P[Z_{n_i} = \varepsilon_i, \forall n_i \in I \setminus K_0;\ Y_{n_i} = 0, \forall n_i \in K_0]} = \frac{\mathbb P[X_{n_{j+1}} = 0 \mid Z_{n_i} = \varepsilon_i, \forall n_i \in I \setminus K_0]}{\mathbb P[Y_{n_i} = 0, \forall n_i \in K_0]},
	\end{equation}
	where we used that $(Y_{n_i})_{n_i\in K_0}$ is independent of $(Z_{n_i})_{n_i \in I \setminus K_0}$. Now, by Lemma \ref{lem:positive-association-k-independence}, 
	\begin{equation}\label{eq:5}
		\mathbb P[X_{n_{j+1}} = 0 \mid Z_{n_i} = \varepsilon_i, \forall n_i \in I \setminus K_0] \le \mathbb P[X_{n_{j+1}} = 0]
	\end{equation}
	since the event $\{Z_{n_i} = \varepsilon_i, \forall n_i \in I \setminus K_0\}$ is $\cl_k(\{n_{j+1}\})$-increasing. Plugging \eqref{eq:5} into \eqref{eq:4} and using the upper bound on the size of $K_0$, we deduce 
	\begin{equation}\label{eq:6}
		\mathbb P[X_{n_{j+1}} = 0 \mid Z_{n_1} = \varepsilon_1,\ldots,Z_{n_j}=\varepsilon_j] \le \frac{1-p}{(1-p')^{2(\Delta-1)^k}}.
	\end{equation}
	Finally, we define $p'$ by $1-p = (1-p')^{2(\Delta-1)^k +1}$ and $\rho$ by $\rho(p) = (p')^2$. These choices establish \eqref{eq:3} and thereby conclude the proof.
\end{proof}

Controlling the propagation of (negative) information while exploring a cluster in a positively associated, 1-independent percolation model will be crucial to establish the upper bounds in Theorems \ref{thm:oriented-d-large}--\ref{thm:oriented-d-2}. Lemma \ref{lem:positive-association-k-independence} can yield some control for specific types of exploration processes. Before illustrating this, we point out that the ideas developed below are specific to 1-independent models and do not seem to generalize naturally to $k$-independent models. 

Let us consider a tree $T=(V,E)$ with root $o \in V$ and  view it as a directed graph by orienting each edge away from the root (we refer to the oriented edge set as $\vv{E}$). Now, consider the level structure $L = (L_i)_{i\ge 0}$, where 
\begin{equation}\label{eq:tree-level-structure}
	L_i := \big\{v \in V : d(o,v)=i\big\}.
\end{equation}
It naturally partitions the edge set $\vv{E}$ by setting
\begin{equation}
	\vv{E}_{i} := \big\{(u,v) \in \vv{E} : u \in L_{i-1}, v \in L_i\big\}
\end{equation}
for $i \ge 1$. 
The cluster of the root, $\mathcal C_o$, can be explored using a breadth-first search: In the first step, one samples the edges in $\vv{E}_1$ to identify all vertices in the first level $L_1$ belonging to $\mathcal C_o$. In the second step, it  suffices to sample the edges in $\vv{E}_2$ that are incident to the vertices in $L_1 \cap \mathcal C_o$ in order to identify all vertices in the second level $L_2$ belonging to $\mathcal C_o$. One then continues like this in all following steps (and eventually stops if $\mathcal C_o$ is finite). 

When studying Bernoulli percolation on $T$ from the perspective of branching processes, one follows this level-by-level exploration and exploits its Markovian structure ($L_{i+1} \cap \mathcal C_o$ depends on the past only via $L_{i} \cap \mathcal C_o$) which is a consequence of `levelwise independence', i.e.\ the fact that the edges in $\vv{E}_{i+1}$ are independent of the previously explored edges in $\vv{E}_1 \cup \ldots \cup \vv{E}_i$. For a positively associated, 1-independent percolation model $\P$, this is generally not the case. However, the next lemma will show that it is possible to compare the exploration of $\mathcal C_o$ under $\P$ with the exploration under its levelwise independent version $\hat{\P}_L$ that is obtained by sampling the edges in $\vv{E}_i$ according to $\P\vert_{\vv{E}_i}$, independently for each $i\ge 1$. More precisely, we will see that  the law of $\mathcal C_o$ under $\P$ stochastically dominates the law of $\mathcal C_o$ under $\hat{\P}_L$. In particular, if $\hat{\P}_L$ percolates, then $\P$ also percolates. 

In the following sections, we need to deal with directed graphs containing undirected cycles such as $\vv{\Z}^n$, and therefore, we choose a more general formulation of the lemma. Given a directed graph $G=(V,\vv{E})$, we call a partition $L = (L_i)_{i\in \Z}$ of the vertex set a \emph{strict level structure}  if $(\vv{E}_{i})_{i\in \Z}$, defined by 
\begin{equation}
	\vv{E}_{i} := \big\{(u,v) \in \vv{E} : u \in L_{i-1}, v \in L_i\big\},
\end{equation}
is a partition of the edge set.\footnote{In a (standard) level structure, any edge $(u,v)\in \vv{E}$ with starting point $u \in L_i$ has its ending point $v$ either in the same level $L_i$ or in the strictly higher level $L_{i+1}$. Here, we call the level structure \emph{strict} since we require the ending point $v$ to be in the strictly higher level $L_{i+1}$.} For example, the sets $L_i := \{x \in \Z^n : x_1 + \ldots + x_n =  i\}$ form a strict level structure of $\vv{\Z}^n$. 
Given a strict level structure $L$ and a bond percolation model $\P$, we define (as before) its \emph{levelwise independent version} $\hat{\P}_L$  by sampling $\omega\vert_{\vv{E}_i}$ according to $\P\vert_{\vv{E}_i}$, independently for each $i\ge 1$.  

We recall that $\vv{\mathcal{C}}_u(\omega)$ denotes the cluster of $u$, that is the set of vertices which are reachable from $u$ by a (directed) $\omega$-open path. For $U \subseteq V$, we introduce the notation 
\begin{equation}
	\vv{\mathcal{C}}_U(\omega) := \bigcup_{u \in U}	\vv{\mathcal{C}}_u(\omega).
\end{equation}

\begin{lemma}\label{lem:levelwise-decoupling}
	Let $G$ be a locally finite directed graph with a strict level structure $L$. Consider a positively associated, 1-independent bond percolation model $\P$ and its levelwise independent version $\hat{\P}_L$. For any finite subset $U \subseteq L_0$, there exists a coupling of $\omega \sim \P$ and $\hat{\omega}_L \sim \hat{\P}_L$ such that almost surely
	\begin{equation}
		\vv{\mathcal{C}}_U(\hat{\omega}_L) \subseteq \vv{\mathcal{C}}_U(\omega) .
	\end{equation}
\end{lemma}

\begin{proof} 
	Without loss of generality we can assume that $U = L_0$, $L_i$ is finite for $i\ge 0$, and $L_i = \emptyset$ for $i<0$. Indeed, $\vv{\mathcal{C}}_U$ only depends on the edges between vertices that are reachable from $U$.
	The strict level structure $L$ and the set $U=L_0$ being fixed, we continue to write $\hat{\P}$ for $\hat{\P}_L$ and $\vv{\mathcal{C}}$ for $\vv{\mathcal{C}}_{L_0}$. 
	We naturally identify $\vv{\mathcal{C}} \subseteq V$ with a configuration $\lambda \in \{0,1\}^V$ by setting $\lambda(v)=1$ whenever $v \in \vv{\mathcal{C}}$, and write $\mathrm Q$ and $\hat{\mathrm Q}$ for the push-forward measures of $\P$ respectively $\hat{\P}$ under $\vv{\mathcal{C}} : \{0,1\}^{\vv{E}} \to \{0,1\}^V$.
	For simplicity of notation, let us write $L^i := L_0 \cup \ldots \cup L_i$ and $\vv{E}^i := \vv{E}_{1} \cup \ldots \cup \vv{E}_{i}$, and denote the restriction of $\lambda \in \{0,1\}^V$ to $L_i$ (resp. $L^i$) by $\lambda_i$ (resp. $\lambda^i$). By definition, $\lambda^i$ only depends on the edges in $\vv{E}^i$.
	Our goal is to construct a coupling of $\lambda \sim \mathrm Q$ and $\hat{\lambda} \sim \hat{\mathrm{Q}}$ such that almost surely 
	\begin{equation}\label{eq:lemma-levelwise-decoupling-1}
		\hat{\lambda} \le \lambda.
	\end{equation} 
	From there, the lemma directly follows by sampling $\hat{\omega}$ conditional on $\hat{\lambda}$ and $\omega$ conditional on $\lambda$. To establish \eqref{eq:lemma-levelwise-decoupling-1}, we apply the result of Kamae, Krengel and O'Brien \cite{Kamae1977} which we stated as Theorem \ref{thm:Kamae} in Section \ref{sec:background-and-notation}.

	 For $n\ge 1$, we define stochastic kernels $p_n$, $\hat{p}_n$ on $\{0,1\}^{L^{n-1}} \times \{0,1\}^{L_n}$ by 
	\begin{equation}
		p_n(\kappa^{n-1},A) = \mathrm Q[A\mid \lambda^{n-1} = \kappa^{n-1} ] \quad \text{and} \quad \hat{p}_n(\kappa^{n-1},A) = \hat{\mathrm Q}[A\mid \hat{\lambda}^{n-1} = \kappa^{n-1} ]
	\end{equation}
	for a configuration $\kappa^{n-1} \in \{0,1\}^{L^{n-1}}$ and an event $A \subseteq \{0,1\}^{L_n}$. The main step is to show that for every $n\ge 1$ and  every $\kappa^{n-1}$,
	\begin{equation}\label{eq:lemma-levelwise-decoupling-2}
		  p_n(\kappa^{n-1},\boldsymbol{\cdot}) \succeq \hat{p}_n(\kappa^{n-1},\boldsymbol{\cdot}) .
	\end{equation}
	The two kernels being equal for $n=1$, we fix any $n \ge 2$ and any $\kappa^{n-1}$. Denote by $\vv{F}_n \subseteq \vv{E}_n$ the subset of edges $(u,v)$ with $\kappa^{n-1}(u)=1$ and by $K_n \subseteq L_n$ the subset of vertices $v$ with an edge $(u,v) \in \vv{F}_n$. 
	The event $\{\lambda^{n-1} = \kappa^{n-1}\}$, viewed as an event on $\{0,1\}^{\vv{E}^{n-1}}$, is $\cl_1(\vv{F}_n)$-increasing since opening any edge $(u,v) \in \vv{E}_{n-1}$ with $\kappa^{n-1}(v)=1$ favours its occurence. Thus, by Lemma \ref{lem:positive-association-k-independence}, for every $(\vv{F}_n)$-increasing event $A$ supported on $\vv{F}_n$, 
	\begin{equation}
		\P[A\mid \lambda^{n-1} = \kappa^{n-1} ] \ge  \P[A] = \hat{\P}[A]  =\hat{\P}[A\mid \lambda^{n-1} = \kappa^{n-1}],
	\end{equation} 
	where we used in the first equality that $A$ is supported on $\vv{F}_n\subseteq \vv{E}_n$, and in the second equality that the edges in $\vv{E}_n$ are independent of the edges in $\vv{E}^{n-1}$ under $\hat{\P}$. It is straightforward to verify that this implies for every $(K_n)$-increasing event $A$ supported on $K_n$, 
	\begin{equation}
		\mathrm Q[A\mid \lambda^{n-1} = \kappa^{n-1}] \ge \hat{\mathrm Q}[A\mid \hat{\lambda}^{n-1} = \kappa^{n-1} ].
	\end{equation}
	In fact, this statement extends to every ($L_n$-)increasing event supported on $L_n$ since $\lambda^{n-1} = \kappa^{n-1}$ forces $\lambda_n(v) = 0$ for every $v \in L_n \setminus K_n$. We have thus established \eqref{eq:lemma-levelwise-decoupling-2}. 
	
	Finally, we observe that 
	$\hat{p}_n(\iota^{n-1},\boldsymbol{\cdot}) \preceq \hat{p}_n(\kappa^{n-1},\boldsymbol{\cdot})$
	holds for every $\iota^{n-1} \le \kappa^{n-1}$ since the edges in $\vv{E}_n$ are sampled independently of the edges in $\vv{E}^{n-1}$ under $\hat{\P}$. The existence of the desired coupling \eqref{eq:lemma-levelwise-decoupling-1} now follows directly from Theorem \ref{thm:Kamae} (with $\P_0 = \mathrm{Q}_0 = \delta_U$).
\end{proof}

Based on the previous lemma, we obtain an alternative proof of Theorem \ref{thm:trees} for the 1-independent case. To this end, consider an infinite, locally finite tree $T$ with the natural strict level structure $L$ given by \eqref{eq:tree-level-structure}. According to Lemma \ref{lem:levelwise-decoupling}, it suffices to prove that any levelwise independent bond percolation model $\hat{\P}_L \in \mathcal P_a(T,p)$ with marginals $p>\frac{1}{\br(T)}$ percolates. The reduction to levelwise independent models facilitates the computation of the first and
second moment. One then concludes as in Section \ref{sec:trees} by closely following the proof of Theorem 6.2 in \cite{Lyons1990}. 

\section{Two-dimensional Percolation}
\label{sec:two-dimensional-percolation}
\subsection{Lower Bound of Theorem \ref{thm:unoriented-d-2}}

In Section \ref{sec:trees}, we studied percolation on trees and established that every positively associated, finite-range dependent percolation model necessarily percolates if its marginals are strictly larger than those of critical Bernoulli percolation. In the 1-independent case, this implies $p_a^+(T) = p_c(T)$.  We begin this section by proving that this equality can be false on graphs containing cycles. More precisely, we construct a positively associated, 1-independent bond percolation model on $\Z^2$ with marginals strictly larger than $\frac{1}{2}$ that does not percolate, and we thus establish  the lower bound $p^+_a(\Z^2)>p_c(\Z^2)=\frac{1}{2}$ of Theorem \ref{thm:unoriented-d-2}. 

A crucial input is the following lower bound for the critical probability of Bernoulli site percolation on the truncated square lattice $\mathbb{T}$ illustrated in Figure \ref{fig:TruncatedSquare1}. Approximations of $p_c^{\mathrm{site}}(\mathbb{T})$ in \cite{jacobsen2014high} suggest that the actual value is about $0.7297$.

\begin{figure}
  \centering
  \subfloat[The black lattice is the truncated square lattice $\mathbb{T}$. By contracting diagonal edges, we obtain the gray copy of $\mathbb{Z}^2$.]{\includegraphics[width=0.3\textwidth]{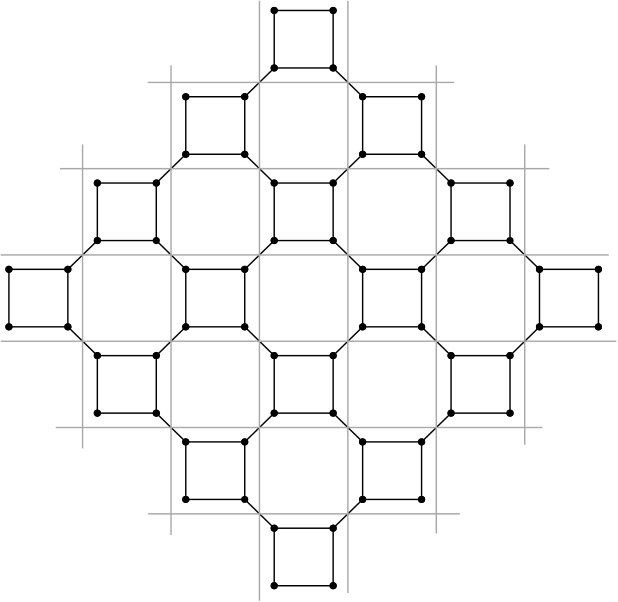}\label{fig:TruncatedSquare1}}
  \hfill
  \subfloat[A bond percolation configuration induced by a site percolation configuration. Open sites resp.\ bonds in blue.]{\includegraphics[width=0.3\textwidth]{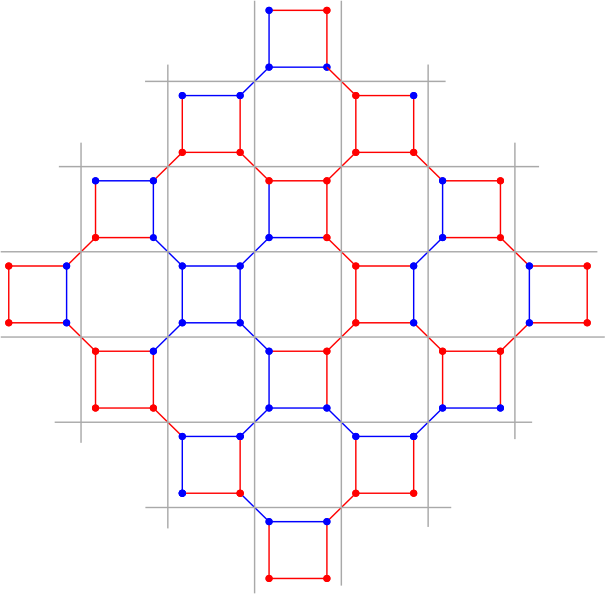}\label{fig:TruncatedSquare2}}
  \hfill
  \subfloat[A bond percolation configuration on $\Z^2$ obtained by contraction.]{\includegraphics[width=0.3\textwidth]{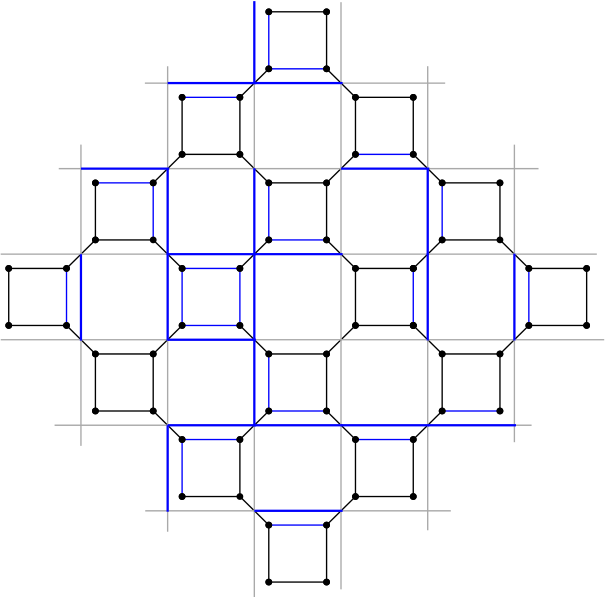}\label{fig:TruncatedSquare3}}
  \caption{Definition of $Q_{p^2} \in \mathcal{P}(\mathbb{Z}^2,p^2)$.}
\end{figure}

\begin{lemma}\label{SitePercolationTruncatedSquareLattice}
    \[p_c^{\mathrm{site}}(\mathbb{T}) > \frac{1}{\sqrt{2}}\]
\end{lemma}

The proof of Lemma \ref{SitePercolationTruncatedSquareLattice} follows closely the strategy presented in \cite[Theorem 3.7]{grimmett1999percolation}, where it is proven that the critical probability of Bernoulli bond percolation on the triangular lattice is strictly smaller than $p_c(\mathbb{Z}^2)$. This strategy is based on a general technique, known as essential enhancements, due to Aizenman and Grimmett \cite{aizenman1991strict}. However,  it was pointed out in  \cite{balister2014essential} that a key combinatorial lemma of the argument in \cite{aizenman1991strict} has a mistake that is not easily rectifiable in full generality. For this reason, we decided to provide some details of the argument in our specific setting, and then refer to\cite[Theorem 3.7]{grimmett1999percolation}.

\begin{proof}[Proof of Lemma \ref{SitePercolationTruncatedSquareLattice}]
   We begin with an observation about Bernoulli bond percolation $\pi_p$ on $\mathbb{Z}^2$. Consider the graph $\text{Subdiv}(\mathbb{Z}^2)$ that is obtained by adding a vertex in the middle of each edge of $\Z^2$ and thereby subdividing each edge into two edges. Clearly, $p_c(\mathbb{Z}^2)$ is equal to $p_c(\text{Subdiv}(\mathbb{Z}^2))^2$. Applying a bond-to-site transformation to $\text{Subdiv}(\mathbb{Z}^2)$, we obtain its line graph, denoted by $\mathbb{L}$, and we obtain  $p_c(\text{Subdiv}(\mathbb{Z}^2)) = p^{\mathrm{site}}_c(\mathbb{L})$. 
   We make the geometric observation that $\mathbb{L}$ can be obtained from $\mathbb{T}$ by adding the two diagonal edges to each square in $\mathbb{T}$. It is thus clear that $p^{\mathrm{site}}_c(\mathbb{T}) \ge p^{\mathrm{site}}_c(\mathbb{L}) = \sqrt{p_c(\Z^2)} = \frac{1}{\sqrt{2}}$, and we are left with proving that the inequality is strict.
   
   As in \cite[Section 3.1]{grimmett1999percolation}, we introduce a two-parameter family of site percolation models to interpolate between $\mathbb{T}$ and $\mathbb{L}$. 
   Denote by $\widetilde{\mathbb{T}}$ the graph obtained by adding in the center of each square in $\mathbb{T}$ a vertex and connecting it to all four vertices of the square. Declare each vertex in $\mathbb{T}$ to be open with probability $p \in [0,1]$ and each \emph{additional} vertex in $\widetilde{\mathbb{T}}$ to be open with probability $s\in[0,1]$, independently of each other, and write $\pi^{\mathrm{site}}_{p,s}$ for the  resulting model. It is straightforward to see that $\pi^{\mathrm{site}}_{p,0}$ (resp.\ $\pi^{\mathrm{site}}_{p,1}$) corresponds to Bernoulli site percolation $\pi^{\mathrm{site}}_p$ on $\mathbb{T}$ (resp.\ on $\mathbb{L}$). From here, the strict inequality can be established analogously to \cite[Theorem 3.7]{grimmett1999percolation}.
    \end{proof}
    
With these preparations at hand, we may turn to the lower bound for $p_a^+(\mathbb{Z}^2)$.

\begin{proof}[Proof of the lower bound of Theorem \ref{thm:unoriented-d-2}]
    Let $\sigma$ be a site percolation configuration on $\mathbb{T}$. First, $\sigma$ naturally induces a bond percolation configuration $\eta$ on $\mathbb T$ by declaring an edge $\eta$-open if and only if both its endpoints are $\sigma$-open.
    Second, we make the geometric observation that contracting the diagonal edges of $\mathbb{T}$ yields the lattice $\mathbb{Z}^2$ (see Figure \ref{fig:TruncatedSquare1}). Through this contraction, $\eta$ induces a bond percolation configuration $\omega$ on $\Z^2$.

     We now argue that there exists an infinite $\sigma$-open path in $\mathbb T$ if and only if there exists an infinite $\omega$-open path in $\Z^2$.
     On the one hand, an infinite $\sigma$-open path in $\mathbb T$ is also $\eta$-open. The contraction then simply removes all diagonal edges and thus yields an infinite $\omega$-open path in $\Z^2$. On the other hand, given a path $\gamma$ in $\Z^2$, we consider a path $\gamma'$ that is mapped to $\gamma$ under the contraction $\mathbb T \to \Z^2$. Note that the choice of $\gamma'$ is unique up to adding or removing one vertex at the start and at the end of $\gamma'$.   Clearly, $\gamma'$ is infinite if $\gamma$ is infinite, and by construction of $\omega$ from $\sigma$, it follows that  $\gamma'$ must be $\sigma$-open if $\gamma$ is $\omega$-open.  

     Finally, it is straightforward to verify that the law of $\omega$ belongs to $\mathcal{P}_a(\Z^2, p^2)$ if $\sigma$ is sampled according to Bernoulli site percolation $\pi_p^{\mathrm{site}}$ on $\mathbb T$. This implies 
     \begin{equation}
         p_a^+(\Z^2) \ge \left(p_c^{\mathrm{site}}(\mathbb T)\right)^2,
     \end{equation}
    and so by Lemma \ref{SitePercolationTruncatedSquareLattice}, we conclude that $p_a^+(\Z^2) > \frac{1}{2}$.
\end{proof}

We remark that the previously mentioned result in \cite{jacobsen2014high} suggests
$p^+_a(\mathbb{Z}^2) \ge 0.5324$.

\subsection{Stochastic Domination for Oriented Percolation Clusters}\label{sec:StochasticDominationOrientedClusters}

In this subsection, we study oriented percolation on $\vv{\Z}^2$ with  strict level structure $L$ given by the diagonals $L_i := \{(x,y) \in \Z^2 : x+y=i \}$. Lemma \ref{lem:levelwise-decoupling} established that for any positively associated, 1-independent bond percolation model $\P$, the cluster of the origin $\vv{\mathcal C}_0(\omega)$ stochastically dominates $\vv{\mathcal C}_0(\hat{\omega}_L)$, where $\omega$ is sampled according to $\P$ and $\hat{\omega}_L$ is sampled according to its levelwise-independent version $\hat{P}_L$. Therefore, it suffices to consider levelwise-independent bond percolation models when studying the parameter  $p_a^+(\vv{\Z}^2)$. This is useful since $(\vv{\mathcal C}_0(\hat{\omega}_L) \cap L_i)_{i\ge 0}$ is a Markov chain with initial state $\{0\}$ whose transitions from a subset of $L_{i-1}$ to a subset of $L_{i}$ are determined by some $\P_{i} \in \mathcal P_a(\vv{E}_i,p)$, and the problem is thus reduced to the study of positively associated, 1-independent bond percolation models on $\vv{E}_1$. However, this is still a large class of models. To make headway, we introduce a new Markov chain $(\mathcal X_i)_{i\ge 0}$ with explicit transition probability, and compare it with $(\vv{\mathcal C}_0(\hat{\omega}_L) \cap L_i)_{i\ge 0}$  in Lemma \ref{lem:markov-chain-lemma-1}.

Let $G=(V,\vv{E})$ be a subgraph of $\vv{\Z}^2$ with  strict level structure $L$ given by $L_i := \{(x,y) \in V : x+y=i \}$, and let $p \in [0,1]$. On the countable state space 
\begin{equation}\label{eq:state-space-X}
	\mathcal S := \bigcup_{i\ge 0} \{W \subseteq L_i : \abs{W}<\infty\},
\end{equation}
we define the transition probability $p_X = (p_X(W,W'))_{W,W' \in \mathcal S}$ as follows: First, we set $p_X(W,W') = 0$ if $W' \nsubseteq N_+(W)$. Now, assume that $W' \subseteq N_+(W)$. In the special case that $W$ is an interval of length $\ell\ge 1$, that is, a subset of the form $\{(x+j,y-j):1 \le j \le \ell\}$, we define
\begin{equation}
	p_X(W,W') = \begin{cases}
		(1-p)^\ell &\text{if}\  \abs{N_+(W)} = \ell + 1,\ W' = \emptyset,\\
		0 &\text{if}\  \abs{N_+(W)} = \ell + 1,\ W' = \{(x+\ell+1,y-\ell)\},\\
		p^{\abs{W'}} (1-p)^{\abs{N_+(W)\setminus W'}}  &\text{otherwise}.
	\end{cases}
\end{equation}
For the general case, we partition $W$ into maximal intervals $W_1,\ldots,W_m$. Noting that the sets $N_+(W_1),\ldots,N_+(W_m)$ are disjoint, we can then partition $W'$ into $W_1',\ldots,W_m'$ defined by $W_k' := N_+(W_k) \cap W'$. We then define
\begin{equation}\label{eq:markov-chain-product-form}
	p_X(W,W') = \prod_{k=1}^{m} p_X(W_k,W_k'),
\end{equation}
which includes the case $p_X(\emptyset,\emptyset)=1$. Throughout the rest of this section, we write $(\mathcal X_i)_{i\ge 0}$ for a Markov chain on $\mathcal S$ with transition probability $p_X$. The following remark will be useful in Section \ref{subsec:upper-bound-unoriented-2-d}.

\begin{remark}\label{rem:markov-chain-increasing-function}
    $(\mathcal X_i)_{i\ge 0}$ can be constructed as an increasing function of a family of independent Ber($p$)-distributed random variables $(Z(v))_{v\in V}$. Indeed, handling separately each interval $W$ of length $\ell \ge 1$, set $W' = \emptyset$ if $\abs{N_+(W)} = \ell + 1$ and $Z(v)=0$ for every $v \in N_+(W) \setminus \{(x+\ell+1,y-\ell)\}$, and set $W' = \{ v \in N_+(W): Z(v)=1\}$ otherwise.
\end{remark}

\begin{lemma}\label{lem:markov-chain-lemma-1}
	Let $G=(V,\vv{E})$ be a subgraph of $\vv{\Z}^2$ with  strict level structure $L$ given by $L_i := \{(x,y) \in V : x+y=i \}$. Consider the levelwise independent version $\hat{\P}_L$ of some $\P \in \mathcal{P}_a(G,p)$. For any finite subset $U \subseteq L_0$, there exists a coupling of  $\hat{\omega}_L \sim \hat{\P}_L$ and $(\mathcal X_i)_{i\ge 0}$ with initial state $U$ and transition probability $p_X$ such that almost surely
	\begin{equation}
		\bigcup_{i\ge 0} \mathcal X_i \subseteq \vv{\mathcal{C}}_U(\hat{\omega}_L) .
	\end{equation}
\end{lemma}

\begin{proof}
	The strict level structure $L$ being fixed, we write $\hat{\P}$ for $\hat{\P}_L$ and $\hat{\omega}$ for $\hat{\omega}_L$. The levelwise independence of $\hat{\P}$ implies that $(\vv{\mathcal C}_U(\hat{\omega}) \cap L_i)_{i\ge 0}$ is a Markov chain on $\mathcal S$ with initial state $U$. We denote its transition probability by $\hat{p}$. We first observe that $\hat{p}(W,W')>0$ only if $W' \subseteq N_+(W)$. Moreover, 1-independence of $\hat{\P}_L$ implies $\hat{p}(W,W') = \prod_{k=1}^{m} \hat{p}(W_k,W_k')$, where $W_1,\ldots,W_k$ are maximal intervals and $W_1',\ldots,W_k'$ are defined as in \eqref{eq:markov-chain-product-form}. Now, assume that $W$ is an interval of length $\ell \ge 1$. The main step of the proof is to show that 
	\begin{equation}\label{eq:lemma-3-main-step}
		p_X(W,\boldsymbol{\cdot})\preceq \hat{p}(W,\boldsymbol{\cdot}).
	\end{equation}
	
	We proceed by constructing a monotone coupling of $W'_X \sim p_X(W,\boldsymbol{\cdot})$ and $\hat{W}' \sim \hat{p}(W,\boldsymbol{\cdot})$. Let $\vv{F}$ denote the edges from $W$ to $N_+(W)$. First, we consider the case $\abs{N_+(W)} \le \abs{W}$. It is then possible to pick for each vertex $v \in N_+(W)$ a vertex $u_v \in W$ in such a way that $(u_v,v) \in \vv{F}$ and no vertex in $W$ is picked twice. Under $\hat{\P}$, the random variables $(\hat{\omega}((u_v,v)))_{v \in N_+(W)}$ are thus independent, and we obtain a monotone coupling by defining $W_X' := \{v \in N_+(W) : \hat{\omega}((u_v,v))=1\}$ and $\hat{W}':= \{v \in N_+(W) : \exists u\in W\ \text{s.t.}\ \hat{\omega}((u,v))=1\}$.
	
	Second, we consider the case $\abs{N_+(W)} = \abs{W}+1$. By increasing $x$-coordinate, we label the vertices in $W$ as $u_1,\ldots,u_\ell$ and the vertices in $N_+(W)$ as $v_1,\ldots,v_{\ell+1}$, and we write $e_{j,k}:=(u_j,v_k)$ with $1 \le j \le \ell$ and $k\in \{j,j+1\}$ for the edges in $\vv{F}$. We sample $\hat{\omega}\sim \hat{\P}$ on $\vv{F}$  and define
	\begin{equation}
		j^\star := \min\{j: \hat{\omega}(e_{j,j})=1\} \in \{1,\ldots,\ell\} \cup \{\infty\}. 
	\end{equation}
	Additionally, we sample an independent Bernoulli random variable $Y$ with parameter $p/\hat{\P}[\hat{\omega}(e_{j^\star,j^\star+1}) = 1 \mid \hat{\omega}(e_{j^\star,j^\star})=1] \in [p,1]$. We define $\hat{W}'$ as before and $W_X'$ by
	\begin{equation}
	v_j \in W_X' \iff \begin{cases} \hat{\omega}(e_{j,j})=1 &\text{if}\ j \le j^\star, \\
		\hat{\omega}(e_{j-1,j}) \cdot Y =1 &\text{if}\ j = j^\star + 1, \\
		\hat{\omega}(e_{j-1,j}) =1 &\text{if}\ j > j^\star + 1. \\
	\end{cases} 
	\end{equation}
	Note that $j = \infty$ implies $W'_X = \emptyset$. It is straightforward to check that $W'_X \sim p_X(W,\boldsymbol{\cdot})$ and that the constructed coupling is monotone. 
	
	Having established \eqref{eq:lemma-3-main-step}, the simple observation that $\hat{p}(W,\boldsymbol{\cdot}) \preceq \hat{p}(Z,\boldsymbol{\cdot})$ if $W \subseteq Z$ allows to deduce that $p_X(W,\boldsymbol{\cdot}) \preceq\hat{p}(Z,\boldsymbol{\cdot})$ for every $W \subseteq Z \in \mathcal S$. Applying Theorem \ref{thm:Kamae} in the Markovian setting concludes the proof.
\end{proof}

In the next lemma, we compare the Markov chain $(\mathcal X_i)_{i\ge 0}$ with a simple bond percolation model defined as follows. Let $(Z^+(v),Z^-(v))_{v \in \Z^2}$ be independent Ber($p$)-distributed random variables and define $\eta \in \{0,1\}^{\vv{E}}$ by setting $\eta((u,v)) = Z^+(u) \cdot Z^-(v)$. We write $\pi_{p,p}$ for the law of $\eta$, and note that $\pi_{p,p} \in \mathcal P(\vv{\Z}^2,p^2)$ is levelwise independent. It is easily seen that $\pi_{p,p}$ percolates if and only if $p^2 > p_c^{\mathrm{site}}(\vv{\Z}^2)$.
\begin{lemma}\label{lem:markov-chain-lemma-2}
	Consider $\vv{\Z}^2$ with strict level structure $L$ given by $L_i := \{(x,y) \in \Z^2 : x+y=i \}$. There exists a coupling of $\eta \sim \pi_{p,p}$  and the Markov chain $(\mathcal X_i)_{i\ge 0}$ with initial state $\{0\}$ such that almost surely
	\begin{equation}
		\vv{\mathcal{C}}_0(\eta) \subseteq \bigcup_{i\ge 0} \mathcal X_i  .
	\end{equation}
\end{lemma}

\begin{proof}
	The levelwise independence of $\pi_{p,p}$ implies that $(\vv{\mathcal C}_0(\eta) \cap L_i)_{i\ge 0}$ is a Markov chain on $\mathcal S$ with initial state $\{0\}$. Denote its transition probability by $p_\eta$. As in the proof of Lemma \ref{lem:markov-chain-lemma-1}, the main step is to show that for any interval $W$ of length $\ell \ge 1$,
	\begin{equation}\label{eq:lemma-4-main-step}
		p_\eta(W,\boldsymbol{\cdot}) \preceq p_X(W,\boldsymbol{\cdot}),
	\end{equation}
	and we proceed by constructing a monotone coupling of $W'_\eta \sim p_\eta(W,\boldsymbol{\cdot})$ and $W_X' \sim p_X(W,\boldsymbol{\cdot})$ based on  independent Ber($p$)-distributed random variables $(Z^+(v),Z^-(v))_{v \in \Z^2}$. Naturally, we  define $W'_\eta := \{v \in N_+(W): \exists u \in W\ \text{s.t.}\ Z^+(u)\cdot Z^-(v)=1\}$. In the case $\abs{N_+(W)} \le \abs{W}$, a monotone coupling is then obtained by defining $W_X':= \{v \in N_+(W) : Z^-(v)=1\}$. 
	
	In the case $\abs{N_+(W)} = \abs{W}+1$, we label the vertices in $W$ as $u_1,\ldots,u_\ell$ and the vertices in $N_+(W)$ as $v_1,\ldots,v_{\ell+1}$, by increasing $x$-coordinate. We define
	\begin{equation}
		j^\star := \min\{j: Z^+(u_j) = 1\} \in \{1,\ldots,\ell\} \cup \{\infty\}, 
	\end{equation}
	and then $W_X'$ by 
	\begin{equation}
		v_j \in W_X' \iff \begin{cases} Z^+(u_j) = 1 &\text{if}\ j \le j^\star, \\
			Z^-(v_j) = 1 &\text{if}\ j > j^\star. \\
		\end{cases} 
	\end{equation}
	Note that $j = \infty$ implies $W'_X = \emptyset$. It is straightforward to check that $W'_X \sim p_X(W,\boldsymbol{\cdot})$ and that the constructed coupling is monotone. 
	
	Having established \eqref{eq:lemma-4-main-step}, the case of general $W$ follows as in the proof of Lemma \ref{lem:markov-chain-lemma-1} and the simple observation that $p_\eta(Z,\boldsymbol{\cdot}) \preceq p_\eta(W,\boldsymbol{\cdot})$ if $Z \subseteq W$ allows to deduce  $p_\eta(Z,\boldsymbol{\cdot}) \preceq p_X(W,\boldsymbol{\cdot})$ for every $Z \subseteq W \in \mathcal S$.  Applying Theorem \ref{thm:Kamae} in the Markovian setting concludes the proof.
\end{proof}

\begin{proof}[Proof of Theorem \ref{thm:oriented-d-2}]
	For the lower bound, we note that a site percolation model on a directed graph $G=(V,\vv{E})$ naturally induces a bond percolation model on $G$ by declaring an edge $e=(u,v)$ open if $u$ is open in the site percolation. For an infinite, connected, locally finite graph $G$, it is immediate that the induced bond model percolates if and only if the site model percolates. Thus, $p_a^+(G) \ge p_c^{\mathrm{site}}(G)$.
	
	For the upper bound, consider the strict level structure $L$ given by $L_i := \{(x,y) \in \Z^2 : x+y=i\}$. Fix any $\P \in \mathcal{P}_a(\vv{\Z}^2,p)$ and its levelwise independent version $\hat{\P}_L$. It follows from Lemma \ref{lem:levelwise-decoupling}, Lemma \ref{lem:markov-chain-lemma-1} and Lemma \ref{lem:markov-chain-lemma-2} that
	\begin{equation}
		\P[0 \rightarrow \infty] \ge \hat{\P}_L[0\rightarrow \infty] \ge \mathbb P[\mathcal{X}_i \neq \emptyset, \forall i\ge 0 \mid \mathcal{X}_0 = \{0\}] \ge \pi_{p,p}[0 \rightarrow \infty].
	\end{equation}
	Thus, $\P$ percolates if $p^2 > p_c^{\mathrm{site}}(\vv{\Z}^2)$, and so we conclude that $p_a^+(\vv{\Z}^2) \le \sqrt{p_c^{\mathrm{site}}(\vv{\Z}^2)}$.
\end{proof}

\begin{remark}
    In view of the previous proof, it is natural to wonder whether the stochastic domination between the laws of $\vv{C}_0(\hat{\omega}_L)$ and $\vv{C}_0(\eta)$ is specific to the graph $\vv{\Z}^2$ or holds in greater generality. The following example shows that the stochastic domination is false if the directed graph $G$ contains a vertex $u$ with outdegree at least four. Restricted to $u$ and its four outgoing edges $e_1,\ldots,e_4$, a model $\P \in \mathcal P_a(G,p)$ is simply a positively associated probability measure on $\{0,1\}^{\{e_1,\ldots,e_4\}}$ with marginals $p$. Consider $\P$ such that $\omega(e_1)=\omega(e_2)$ and $\omega(e_3)=\omega(e_4)$ are sampled independently. Then the increasing event $A := \{\omega(e_1)=\omega(e_3)=1\} \cup \{\omega(e_2)=\omega(e_4)=1\}$ has probability $\P[A]=p^2$. However, $\pi_{p,p}[A] = 2p^3 - p^5$, which is strictly larger than $p^2$ if $p \in ((\sqrt{5}-1)/2,1)$.
\end{remark}

\subsection{Upper Bound of Theorem \ref{thm:unoriented-d-2}}
\label{subsec:upper-bound-unoriented-2-d}

In this subsection, we will obtain the upper bound of Theorem \ref{thm:unoriented-d-2} through a static renormalization. We refer the reader to \cite[Section 7.4]{grimmett1999percolation} for background on the use of renormalization techniques in percolation.
For any percolation model $\P_0 \in \mathcal{P}_a(\mathbb{Z}^2, p_0)$, where $p_0= 0.77$, we will construct a sequence of percolation models $\P_1 \in \mathcal{P}_a(\mathbb{Z}^2, p_1), \P_2 \in \mathcal{P}_a(\mathbb{Z}^2, p_2),\ldots $ such that 
\begin{equation}
\lim_{n \to \infty} p_n = 1, \quad  \text{and} \quad \forall n\ge 0,\ \P_{n+1}\ \text{percolates} \implies \P_n\ \text{percolates}.
\end{equation}
It is a standard consequence of Peierl's argument that for sufficiently small $\epsilon>0$, any 1-independent percolation model with marginals $1-\epsilon$ must percolate. Hence, the existence of such a sequence of percolation models will imply that $\P_0$ percolates, and thus establish the upper bound $p_a^+(\Z^2) \le 0.77$ of Theorem \ref{thm:unoriented-d-2}.

Unlike in most static renormalizations on $\mathbb{Z}^2$, we will consider crossing events in diagonal boxes instead of axis-parallel boxes. We fix the directions $e_1 = (1,1)$ and $e_2=(-1,1)$, and define the map $\phi: \Z^2 \to \Z^2$ by $\phi(x,y) := (w+1)\cdot (x e_1 + y e_2)$ for some fixed $w \in \Z_+$.  Let us now tile $\Z^2$ with diagonal squares given by 
\begin{equation}
    S(x,y) := \{v \in \Z^2 : d(\phi(x,y),v) \le w\}, \ \text{where}\ (x,y) \in \Z^2.
\end{equation}
We note that not every vertex is contained in one of the diagonal squares. Indeed, any two distinct squares are at distance at least 2 from each other. This will be important to preserve 1-independence under the renormalization scheme.

Next, let us introduce diagonal boxes for the relevant crossing events. For any $u \in \mathbb{Z}^2$ and $w,\ell \ge 0$, we define the diagonal boxes $\mathcal{B}_1(u;w,\ell)$ and $\mathcal{B}_2(u;w,\ell)$  as the respective subgraphs induced by the vertex sets
\begin{equation} \label{eq:definition-diagonal-box}
    \bigcup_{k=0}^\ell \{v \in \Z^2 :  d(u+ k \cdot e_i,v) \le w\} \ \text{for}\ i=1\ \text{resp.}\ i=2.
\end{equation}
If we orient the diagonal box $\mathcal{B}_i(u;w,\ell)$ in the direction $e_i$, we obtain a directed graph $\vv{\mathcal{B}}_i(u;w,\ell)$ that naturally partitions into $2(\ell + w) + 1$ levels with even (resp.\ odd) levels of size $w +1$ (resp.\ $w$). Importantly, this level structure, denoted by $(L_i)_{i=0}^{2(\ell+w)}$, is strict which will allow us to apply Lemma \ref{lem:markov-chain-lemma-1}. This is the primary reason for considering diagonal boxes instead of axis-parallel ones.  
    
    \begin{figure}
      \centering
      \subfloat[A crossing of $\mathcal{B}_1(\phi(x,y);w,w+1)$ in the $e_1$-direction]{\includegraphics[width=0.46\textwidth]{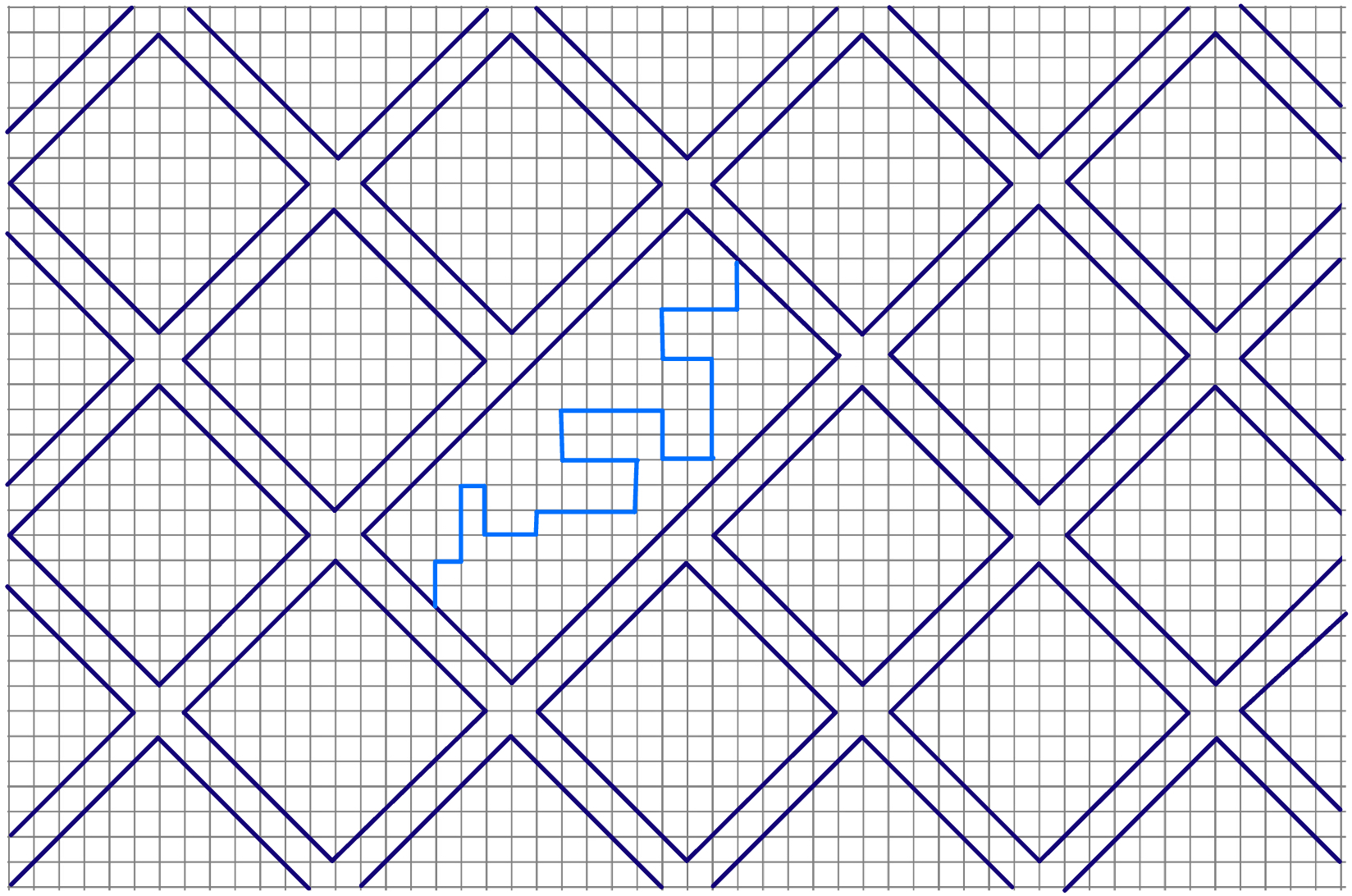}
      \label{fig:Renormalization1}}
      \hfill
      \subfloat[A crossing of $\mathcal{B}_2(\phi(x+1,y);w,0)$ in the $e_2$-direction]{\includegraphics[width=0.46\textwidth]{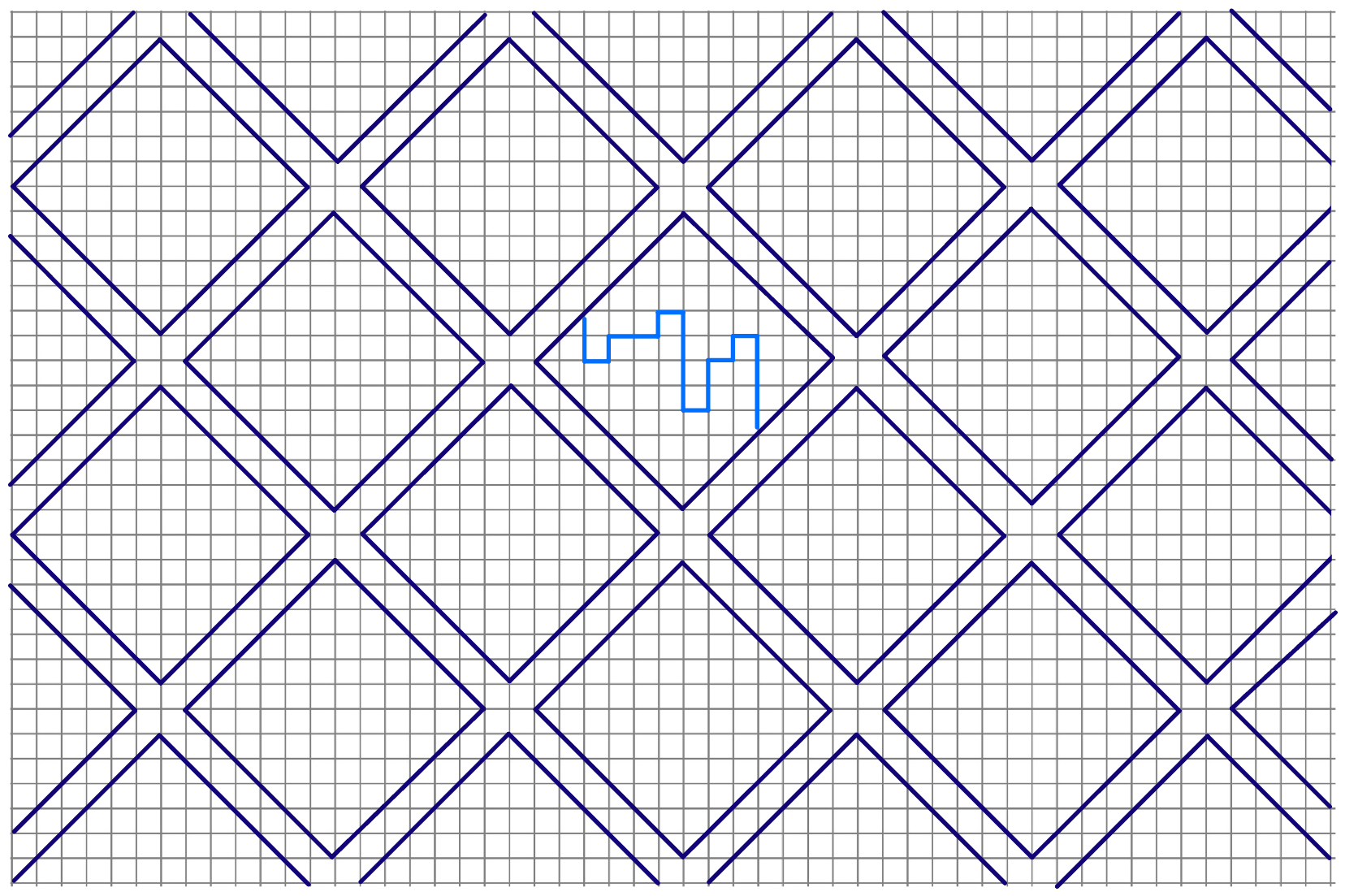}
      \label{fig:Renormalization2}}
      \caption{The edge $e=\{(x,y),(x+1,y)\}$ is declared $\omega_{n+1}$-open if the crossing events in (a) and (b) occur in $\omega_n$.}
    \end{figure}

We are now in position to describe the renormalization scheme, i.e.\ how to define the percolation model $\P_{n+1}$  based on $\P_n$. To this end, let $\omega_n$ be a percolation configuration on the underlying square lattice sampled according to $\P_n$. We then declare an edge $e=\{(x,y),(x+1,y)\}$ of the renormalized square lattice to be $\omega_{n+1}$-open if the diagonal long rectangle $\mathcal{B}_1(\phi(x,y);w,w+1)$ contains an $\omega_n$-open crossing in the $e_1$-direction (see Figure \ref{fig:Renormalization1}) and the diagonal square $\mathcal{B}_2(\phi(x+1,y);w,0)$ contains an $\omega_n$-open crossing in the $e_2$-direction (see Figure \ref{fig:Renormalization2}). 
Similarly, we declare an edge $e=\{(x,y),(x,y+1)\}$ to be $\omega_{n+1}$-open if the diagonal long rectangle $\mathcal{B}_2(\phi(x,y);w,w+1)$ contains an $\omega_n$-open crossing in the $e_2$-direction and the diagonal square $\mathcal{B}_1(\phi(x,y+1);w,0)$ contains an $\omega_n$-open crossing in the $e_1$-direction. 
This defines the percolation configuration $\omega_{n+1}$ based on $\omega_n$, and we write $\P_{n+1}$ for its law. It is standard and easy to check that $\P_{n}$ percolates if $\P_{n+1}$ percolates. Moreover, $\P_{n+1}$ inherits positive association and 1-independence from $\P_n$ since crossing events are increasing and the state of any two vertex-disjoint edges in the renormalized lattice depends on two vertex-disjoint sets of edges in the underlying lattice.

It remains to lower bound the marginal $p_{n+1}$ of the percolation model $\P_{n+1}$. By construction and positive association, 
\begin{align}
    \P_{n+1}[\{(x,y),(x+1,y)\}\ \text{open}] \ge &\P_n[\mathcal{B}_1(\phi(x,y);w,w+1) \ \text{crossed in}\ e_1\text{-direction} ] \\
    &\cdot \P_n[\mathcal{B}_2(\phi(x+1,y);w,0) \ \text{crossed in}\ e_2\text{-direction} ].
\end{align} 
In order to lower bound these two crossing probabilities, we orient the edges of the diagonal long rectangle $\mathcal{B}_1(\phi(x,y);w,w+1)$ in the $e_1$-direction and the edges in the diagonal square $\mathcal{B}_2(\phi(x+1,y);w,0)$ in the $e_2$-direction, and restrict our attention to oriented crossings. Recall the strict level structure $(L_i)_{i=0}^{4w+2}$ of $\vv{\mathcal{B}}_1(\phi(x,y);w,w+1)$ and $(L_i)_{i=0}^{2w}$ of $\vv{\mathcal{B}}_2(\phi(x+1,y);w,0)$ introduced below \eqref{eq:definition-diagonal-box}.  Applying Lemma \ref{lem:markov-chain-lemma-1} with $U=L_0$, we obtain 
\begin{equation}
    \P_n[\vv{\mathcal{B}}_1(\phi(x,y);w,w+1) \ \text{crossed in}\ e_1\text{-direction} ] \ge \mathbb P[\mathcal X_{4w + 2} \neq \emptyset \mid \mathcal X_0  =L_0] =: q_{4w+2}(p_n)
\end{equation}
and
\begin{equation}
    \P_n[\vv{\mathcal{B}}_2(\phi(x+1,y);w,0) \ \text{crossed in}\ e_2\text{-direction} ] \ge \mathbb P[\mathcal X_{2w} \neq \emptyset \mid \mathcal X_0  =L_0] =: q_{2w}(p_n),
\end{equation}
where we chose to lower bound both probabilities with respect to the same Markov chain as the diagonal boxes in the $e_1$-direction and the $e_2$-direction are identical up to reflection along the vertical axis. Moreover, we chose to make explicit the dependence of the right-hand side on $p_n$ (via the transition probability of  the Markov chain). In summary, we have argued that 
\begin{equation}
    \P_{n+1}[\{(x,y),(x+1,y)\}\ \text{open}] \ge q_{4w+2}(p_n) \cdot q_{2w}(p_n),
\end{equation} 
and analogously, one obtains the same bound for edges of the form $\{(x,y),(x,y+1)\}$. Thus, we have shown that for every $n\ge 0$, $\P_{n+1} \in \mathcal{P}_a(\Z^2,p_{n+1})$, where 
\begin{equation}
    p_{n+1} := q_{4w+2}(p_n) \cdot q_{2w}(p_n),
\end{equation}
and it remains to prove that $p_0 = 0.77$ implies
\begin{equation}\label{eq:asymp-survival}
        \lim_{n \rightarrow \infty} p_{n} = 1.
\end{equation}
In fact, using a standard result of reliability theory (see, e.g., \cite{durrett1992stochastic} or Section 2.5 in \cite{grimmett1999percolation}), it will be sufficient to prove that 
$p_1 > p_0$.
To state this result, we consider a finite product space $\{0,1\}^N$, where $N \in \Z_+$. A coordinate $i \in N$ is called \emph{essential} for an event $A\subset \{0,1\}^N$ if $\mathrm{1}_A(\omega_i) \neq \mathrm{1}_A(\omega^i)$ for some $\omega \in \{0,1\}^N$, where $\omega_i$ (resp.\ $\omega^i)$ is obtained from $\omega$ by setting the coordinate $i$ to $0$ (resp.\ $1$).
\begin{lemma}[Lemma 2 in \cite{durrett1992stochastic}] \label{DurrettIncreasingLemma}
    Let $\pi_p$ be the product measure on $\{0,1\}^N$, and let $A$ be an increasing event with at least two essential coordinates. Writing $h(p):=\pi_p(A)$, we have 
    \begin{equation}
        h(p) > p \quad \implies \quad \lim_{n \rightarrow \infty} h^{(n)}(p) = 1,
    \end{equation}
    where $h^{(n)}$ is the $n$-th composition of $h$.
\end{lemma}
Using Remark \ref{rem:markov-chain-increasing-function}, we write $q_{4w+2}(p)$ and $q_{2w}(p)$ as probabilities of increasing events on two separate product spaces. Taking the product of the two spaces, we then apply Lemma \ref{DurrettIncreasingLemma} to the function $h(p) := q_{4w+2}(p) \cdot q_{2w}(p)$ to conclude that 
 \begin{equation}
        p_1 = q_{4w+2}(p_0) \cdot q_{2w}(p_0) > p_0 \quad \implies \quad \lim_{n \rightarrow \infty} p_n = 1.
    \end{equation}

We are left with computing the survival probabilities $q_{4w+2}(p_0)$ and $q_{2w}(p_0)$. We do so using a computer-assisted computation based on an algorithm that is similar to the algorithm used in \cite{Balister1993}  for computing upper bounds on the critical probability  $p^{\mathrm{site}}_c(\vv{\mathbb{Z}}^2)$ (see also \cite{durrett1992stochastic} and \cite{balister1994improved}). We refer to Appendix \ref{appendix} for an outline of the algorithm. 
The following table shows the results of the computation for $w=20$ and $p_0= 0.77$ as well as for a slightly smaller value of $p_0$.
\begin{figure}[H]
    \centering
    \begin{tabular}{ |p{2cm}||p{4.5cm}|p{4.5cm}|p{2cm}|  }
 \hline
 marginal $p_0$& long rectangle $q_{4w+2}(p_0)$ & square $q_{2w}(p_0)$ & marginal $p_1$\\
 \hline
 $0.767$   & $< 0.7872$ & $< 0.939$  &   $< 0.74$\\
 $0.77$   & $\geq 0.8187$ & $\geq 0.949$  &   $\geq 0.776$\\
 \hline
\end{tabular}
    \caption{Survival Probabilities $q_{4w+2}(p_0)$ and $q_{2w}(p_0)$ with $w=20$}
    \label{fig:Survival_Probabilities}
\end{figure}
The second line establishes $p_1 >p_0$ for $p_0 = 0.77$, and it thus concludes the proof of $p_a^+(\vv{\Z^2}) \le 0.77$, the upper bound of Theorem \ref{thm:unoriented-d-2}. 
The first line  shows that for $w=20$, one cannot hope to improve much on $0.77$. From the results of Figure \ref{fig:Survival_Probabilities}, one might be inclined to think that using larger boxes (i.e.\ increasing the value of $w$) might yield a significant improvement. However, high confidence results in the following table suggest that even for $w=50$, no significant improvement on $0.77$ can be achieved.

\begin{figure}[H]
     \centering
    \begin{tabular}{ |p{2cm}||p{4.5cm}|p{4.5cm}|p{2cm}|  }
 \hline
 marginal $p_0$& long rectangle $q_{4w+2}(p_0)$ & square $q_{2w}(p_0)$ & marginal $p_1$\\
 \hline
 $0.76$   & $< 0.79$ & $< 0.96$  &   $< 0.7584$\\
 $0.77$   & $\geq 0.817$ & $\geq 0.948$  &   $\geq 0.774516$\\
 \hline
\end{tabular}
    \caption{$99\%$-confidence results for $q_{4w+2}(p_0)$ and $q_{2w}(p_0)$ with $w=50$}
    \label{fig:HighConfSurvival_Probabilities}
\end{figure}
To obtain the high confidence results of Figure \ref{fig:HighConfSurvival_Probabilities}, we sampled the Markov chain $(\mathcal{X}_i)_{i\ge 0}$ using the construction as an increasing function of independent, Ber($p$)-distributed random variables (see Remark \ref{rem:markov-chain-increasing-function}).

We conclude this section by pointing out that using the tiling of $\vv{\Z}^2$ described in \cite{durrett1992stochastic}, one could obtain an upper bound on $p^+_a(\vv{\Z}^2)$ in a similar fashion. This would yield a slight improvement of the bound in Theorem \ref{thm:oriented-d-2}.

\section{High-dimensional Oriented Percolation}
\label{sec:high-dimensional-percolation}

In this section, we study oriented percolation on $\vv{\Z}^n$ for large $n$ and present the proof of Theorem \ref{thm:oriented-d-large}. We follow the overall strategy of Bollobás and Kohayakawa \cite{Bollobas1994} and split the proof into two parts. First, we show in Subsection \ref{Sec:ShortTrees} that the local properties of 1-independent, positively associated bond percolation models on $\vv{\Z}^n$ are similar to those on an $n$-ary tree. Second, in Subsection \ref{Sec:Connecting}, we reduce the question about percolation on $\vv{\Z}^n$ to a simple question about percolation on $\vv{\Z}^2$. 

\subsection{Growing Short Trees}\label{Sec:ShortTrees}

Throughout this subsection, we fix the strict level structure $L$ given by the diagonals $L_{i} = \{v \in \vv{\Z}^{n}: v_1 + \ldots + v_{n} = i\}$. 
Using a comparison with a subcritical branching process, the following simple lemma establishes the existence of bond percolation models $\P \in \mathcal{P}(\vv{\Z}^n, \frac{\sqrt{2}}{n+1})$ that do not percolate. In particular, this implies the lower bound of Theorem \ref{thm:oriented-d-large}.

\begin{lemma}\label{SubcriticalBranching}
    For every $p \in [0,1]$ and $n\ge 1$, there exists $\P \in \mathcal{P}(\vv{\Z}^n, p)$ such that for every $i\ge 0$,
    \begin{equation}
        \P[\abs{\vv{\mathcal{C}}_0(\omega) \cap L_{2i}} > 0] \le \left(\frac{n(n+1)}{2}\cdot p^2\right)^i\, . 
    \end{equation}
\end{lemma}

\begin{proof}
    Let $(Z^\text{in}_u, Z_u^\text{out})_{u \in \vv{\Z}^n}$ be a collection of independent Ber($p$)-distributed random variables. We construct $\omega \sim \P \in \mathcal{P}(\vv{\Z}^n, p)$ as follows: For every vertex $u$ in an \emph{even} level $L_{2i}$, $i \in \Z$, we declare all incoming (resp.\ outgoing) edges of $u$ open if and only if $Z^\text{in}_u = 1$ (resp.\ $Z^\text{out}_u = 1$). 
    Let $G(u)$ be the grandchildren of a vertex $u$, i.e.\ the vertices in $N_+^2(u)$ which are reachable from $u$ by an open directed path of length $2$. For $u \in L_{2i}$, $i\in \Z$,
    \begin{equation}
        \mathrm{E}\left[\abs{G(u)}\right] = \sum_{w\in N_+^2(u)} \P\left[ \omega((u,v))=\omega((v,w))=1\ \text{for some}\ v \in N_+(u) \right]  = \frac{n(n+1)}{2} \cdot p^2\,.
    \end{equation}
    A direct comparison with a branching process yields the upper bound 
    \begin{equation}
        \mathrm{E}\left[\abs{\vv{\mathcal{C}}_0(\omega) \cap L_{2i}}\right] \le \left(\frac{n(n+1)}{2} \cdot p^2\right)^i\quad \text{for every}\ i\ge 0\, ,
    \end{equation}
    and the result follows from Markov's inequality.
\end{proof}

To obtain the matching upper bound of Theorem \ref{thm:oriented-d-large}, the key step is to show that in sufficiently high dimension $n$ and for marginals $p$ strictly larger than $\frac{\sqrt{2}}{n}$, the cluster of the origin grows like a supercritical branching process for the first $O(\log(n))$ levels. The following lemma makes this statement precise.

\begin{lemma}\label{TightBranchingProcess}
    For every $\theta \in (0,1)$ and $c \in (0,\infty)$, there exists $C \in (0,\infty)$ such that for $n$ sufficiently large,
    \begin{equation}
        \P \in \mathcal{P}_a\left(\vv{\Z}^n,(1 + \theta)\frac{\sqrt{2}}{n} \right) \implies \P\left[\abs{\vv{\mathcal{C}}_0(\omega) \cap L_{C\log(n)}} \ge n^c \right] \ge \frac{\theta}{40 n^2}.
    \end{equation}
\end{lemma}

\begin{remark}
    To improve the upper bound on $p_a^+(\Z^n)$ of Corollary \ref{thm:unoriented-d-large} by a factor $1/2$, one needs to show that already for marginals $p$ strictly larger than $\frac{1}{\sqrt{2} n}$, the cluster of the origin grows like a supercritical branching process for the first $O(\log(n))$ levels. To this end, one would orient the edges away from the origin and consider the strict level structure given by $L_i = \{v \in \mathbb{Z}^n : d(0,v) = i\}$. For this choice, any vertex $v \in L_i$ has $2n-i$ outgoing edges, which is comparable with $2n$ for $i = O(\log(n))$.
\end{remark}

The rest of this subsection is devoted to proving Lemma \ref{TightBranchingProcess}. By Lemma \ref{lem:levelwise-decoupling}, it is without loss of generality to restrict our attention to levelwise independent bond percolation models. 

\subsubsection{Static Tree Embeddings}

To illustrate the general idea, let $T$ be a tree rooted at $0$ with levels $(T_i)_{i\ge 0}$ that is embedded into $\vv{\Z}^n$ in such a way that $T_i \subseteq L_i$ for every $i\ge 0$. We observe that by restriction, any levelwise independent bond percolation model $\P \in \mathcal{P}_a(\vv{\Z}^n,p)$ naturally induces a bond percolation model on $T$, and thanks to the 1-independence of $\P$, the statuses of two edges $(u_1,v_1)$ and $(u_2,v_2)$ of $T$ are independent if $u_1 \neq u_2$. 
 Given a vertex $v$ in $T$, we denote by $p(v)$ its parent, i.e.\ the unique vertex $u$ such that $(u,v)$ is an edge of $T$. Iteratively, we define the sets of \emph{active} vertices $\A_0 := \{0\}$ and for $i\ge 1$, 
\begin{equation}
    \A_i := \left\{v \in T_{i} : p(v) \in \A_{i-1}\ \text{and}\ \omega((p(v),v)) = 1 \right\}.
\end{equation}
These are the vertices at level $i$ which are in the open cluster of the root $0$ with respect to the restricted bond percolation on $T$. By construction, we have $\A_i \subseteq \mathcal C_0(\omega) \cap L_i$ for every $i\ge 0$, and our strategy will be to prove lower bounds on the size of $\A_i$ in order to establish Lemma \ref{TightBranchingProcess}. 

As a first attempt, we consider a static tree embedding. If $d \cdot \ell \leq n$, it is straightforward to check that a $d$-ary tree $T$ of height $\ell$ and rooted at $0$ can be embedded into $\vv{\Z}^n$ in such a way that $T_i \subseteq L_i$ for every $i\ge 0$. Choosing $d = \frac{n}{C\log(n)}$ and $\ell = C\log(n)$, the Paley-Zygmund inequality yields that for every $\theta \in (0,1)$ and $c \in (0,\infty)$, there exists $C \in (0,\infty)$ such that for every $n$ sufficiently large,
    \begin{equation} \label{eq:short-trees-first-attempt}
        \P \in \mathcal{P}_a\left(\vv{\Z}^n,(1 + \theta)\frac{C\log(n)}{n} \right) \implies \P\left[\abs{\A_{C\log(n)}} \ge n^c \right] \ge \frac{\theta}{4n}.
    \end{equation}
However, we observe that for marginal $p= (1 + \theta)C\log(n)/n$, the expected size of $\A_i$ is only $(dp)^i = (1+\theta)^{i}$, while $T_i$ has size $(n/C\log(n))^{i}$. Hence, much of the space allocated to the static tree embedding of the tree is actually not used. To make better use of this space, we work with dynamic tree embeddings from now on. 

\subsubsection{Dynamic Tree Embeddings}

From now on, we assume without loss of generality that every edge $e \in \vv{E}(\vv{\Z}^n)$ is open with probability $p := (1 + \theta)\frac{\sqrt{2}}{n}$ and that $p \le 2/n$. 
We begin with introducing our setup for dynamic tree embeddings. Iterating over the levels $(L_i)_{i\ge 0}$, we will define sets of \emph{active} vertices $(\A_i)_{i\ge 0}$, sets of \emph{reachable} vertices $(\mathcal R_i)_{i\ge 1}$, and \emph{allocation functions} $(h_i)_{i\ge 1}$ such that for every $i\ge 1$,
\begin{equation}
    \A_i \subseteq h_i(\A_{i-1})\subseteq \mathcal R_i .
\end{equation}
We initialise the construction with $\A_0 := \{0\}$. For $i\ge 1$, given the set of active vertices $\A_{i-1} \subseteq L_{i-1}$ in the previous level, we define the set of reachable vertices in the next level by 
\begin{equation}
    \mathcal R_i := \bigcup_{u \in \A_{i-1}} N_+(u) \subseteq L_i .
\end{equation}
The main step is then to chooses an allocation function $h_i : \A_{i-1} \to \mathcal R_i$ satisfying 
\begin{enumerate}
    \item[(i)] $h_i(u) \in N_+(u)$ for every $u \in \A_{i-1}$,
    \item[(ii)] $h_i(u) \cap h_i(v) = \emptyset$ for every $u \neq v \in \A_{i-1}$.
\end{enumerate}
In words, $h_i$ allocates each reachable vertex in $R_i$ to one (or none) of the vertices in $\A_{i-1}$, from which it can be reached. While setting $h_i(u)=\emptyset$ for each $u \in \A_{i-1}$ would be a valid choice, we will be interested in choosing an allocation function which allocates a large number of reachable vertices to each vertex in $\A_{i-1}$. 
As before, given a vertex $v \in \mathcal R_i$, we denote by $p(v) \in \A_{i-1}$ its parent, i.e.\ the unique vertex $u$ such that $v \in h_i(u)$.
Having chosen the allocation function $h_i$, one is then in position to define the active vertices in level $L_i$ by 
\begin{equation}
    \A_i := \bigcup_{u \in \A_{i-1}} \left\{ v \in h_i(u) : \omega((u,v)) = 1 \right\}. 
\end{equation}
We note that $\A_i \subseteq \mathcal C_0(\omega) \cap L_i$ for every $i\ge 0$. Importantly, the sets of active vertices $(\A_i)_{i\ge 1}$ only depend on the statuses of the edges in
\begin{equation}
    \left\{(u,v) \in \vv{E}(\vv{\Z}^n) : u \in \A_{i-1}\ \text{and}\ v \in h_i(u)\ \text{for some}\ i\ge 1\right\},
\end{equation}
which form a tree by construction. Therefore, whenever $u \neq u'$, the statuses of the edges in $\{(u,v): v \in h_i(u)\}$ are independent of the statuses of the egdes in $\{(u',v'): v' \in h_i(u')\}$. 

\subsubsection{Conflicting Pairs of Type (b) and (c)}

We now focus on the choice of the allocation functions $(h_i)_{i\ge 1}$. Whenever a vertex $v \in \mathcal R_i$ is reachable from more than one vertex in $\A_{i-1}$, say from $u$ and $u'$ (among others), there is a conflict of interest since we must choose whether $v \in h_i(u)$ or to $v \in h_i(u')$. In this case, we call  $(u,u')$ a \emph{conflicting pair} having a \emph{conflict} at $v$. Since $d(u,u')=2$, we have $d(p(u),p(u'))\in \{0,2,4\}$, and it will be useful to categorise conflicting pairs $(u,u')$ according to the distance between their respective parents $p(u),p(u') \in L_{i-2}$. We say that a conflicting pair $(u,u')$ is of type 
\begin{equation}
    \text{(a)}\ \text{if}\ p(u)=p(u'), \quad \text{(b)}\ \text{if}\ d(p(u),p(u'))=2, \ \text{and}\quad \text{(c)}\ \text{if}\ d(p(u),p(u'))=4.
\end{equation}
Note that for a conflicting pair $(u,u')$ of type (c), we have $p(u) + e_i + e_j = p(u') + e_k + e_\ell$ for some $1\le i,j,k,\ell\le n$ with $\{i,j\} \cap \{k,\ell\} = \emptyset$ and thus we must have $u \in \{p(u)+e_i, p(u)+e_j\}$ and $u' \in \{p(u')+e_k, p(u')+e_\ell\}$.
To further split (b) into three subcategories, we note that whenever $d(p(u),p(u'))=2$, there exist $1 \le i \neq j \le n$ such that $p(u) + e_i =  p(u')+ e_j$. We say that a conflicting pair $(u,u')$ is of type 
\begin{enumerate}
    \item[(b.i)] if $u = p(u) + e_i$ and $u' = p(u') + e_k$ for some $k \neq j$,
    \item[(b.ii)] if $u = p(u) + e_k$ and $u' = p(u') + e_j$ for some $k \neq i$,
    \item[(b.iii)] if $u = p(u) + e_k$ and $u' = p(u') + e_k$ for some $k \not\in \{i,j\}$.
\end{enumerate}
We observe that (b.i) and (b.ii) are anti-symmetric in $u,u'$, i.e. if $(u,u')$ is a conflicting pair of type (b.i), then $(u',u)$ is a conflicting pair of type (b.ii). Hence, it will be sufficient to resolve  the conflicting pairs of types (a), (b.ii), (b.iii), and (c) for every vertex $u \in \A_{i-1}$. 

It turns out that conflicting pairs of type (b.ii), (b.iii), and (c) are rare and will thus be easy to deal with.\\

\textbf{Claim 1:} Let $n$ be sufficiently large. With probability $1- \exp(-\sqrt{n})$, it holds that for every $1 \le i \le (\log(n))^2$ and for every active vertex $u \in \A_{i-1}$, the set
\begin{equation}
   \Gamma(u) := \bigcup_{\substack{u' \in \A_{i-1}:\ (u,u')\ \text{conflicting} \\ \text{pair of type (b.ii), (b.iii), or (c)}}} N_+(u) \cap N_+(u')
\end{equation}
has size at most $\sqrt{n}$.
\begin{proof}[Proof of Claim 1]
    Given $\A_{i-2}$, $\mathcal R_{i-1}$ and $h_{i-1}$, we fix a vertex $u \in \mathcal R_{i-1}$. Note that for every $v \in N_+(u)\subseteq L_i \cap (\Z_+)^n$, there are at most $i$ vertices $u' \in L_{i-1} \cap (\Z_+)^n $ with $v \in N_+(u')$, and hence, there are at most $ni$ potentially conflicting pairs $(u,u')$. 
    The key observation is now that every vertex $p' \in \A_{i-2}$ with $p' \neq p(u)$ can only be the parent of at most two vertices $u'$ for which $(u,u')$ is potentially a conflicting pair of type (b.ii), (b.iii) or (c). Indeed, if $d(p(u),p')=2$, there is one vertex $u' \in N_+(p')$ such that $(u,u')$ is potentially a conflicting pair of type (b.ii), and there is one vertex $u' \in N_+(p')$ such that $(u,u')$ is potentially a conflicting pair of type (b.iii).\footnote{However, if $d(p(u),p')=2$, there are $n-1$ vertices $u' \in N_+(p')$ such that $(u,u')$ is potentially a conflicting pair of type (b.i).}
    Otherwise, if $d(p(u),p')=4$, there are at most two vertices $u' \in N_+(p')$ such that $(u,u')$ is potentially a conflicting pair of type (c). 

    Since $h_{i-1}$ is a allocation function, we recall that the outgoing edges of distinct parents in $\A_{i-2}$ are sampled independently. Therefore, if $u \in \A_{i-1}$, the total number of active vertices in $u' \in \A_{i-1}$ for which $(u,u')$ is a conflicting pair of type (b.ii), (b.iii) or (c) is stochastically dominated by 
    \begin{equation}
        X := \sum_{m=1}^{ni} 2 X_m,
    \end{equation}
    where $(X_m)_{m=1}^{ni}$ are independent, Ber($4/n$)-distributed. Here, the factor $2$ accounts for the at most two vertices $u'$ per parent $p'$, and the parameter $4/n$ is an upper bound on the probability that at least one of the two outgoing edges of the parent is open. Using $i \le (\log(n))^2$, a Chernoff bound implies that for $n$ sufficiently large
    \begin{equation}
        \mathbb P[X \ge \sqrt{n}] \le e^{-2\sqrt{n}}.
    \end{equation}
    From there, the claim follows by taking a union bound over all $u \in \mathcal R_{i-1} \subseteq L_{i-1} \cap (\Z_+)^n$ and over all $i \in \{1,\ldots,(\log(n))^2\}$.
\end{proof}

Before continuing with the proof of Lemma \ref{TightBranchingProcess}, let us explain a simple way of dealing with conflicts of type (a) which implies that for every $\theta \in (0,1)$ and $c \in (0,\infty)$, there exists $C \in (0,\infty)$ such that for every $n$ sufficiently large,
    \begin{equation} \label{eq:short-trees-second-attempt}
        \P \in \mathcal{P}_a\left(\vv{\Z}^n,(1 + \theta)\frac{3}{2n} \right) \implies \P\left[\abs{\A_{C\log(n)}} \ge n^c \right] \ge \frac{\theta}{10n}.
    \end{equation}
The idea is to choose allocation functions $(h_i)_{i\ge 1}$ such that for every $i\ge 1$ and $u \in \A_{i-1}$, we have $\abs{h_i(u)}\le 2n/3$. We initialise with $h_1(0) = \{e_1,\ldots,e_{2n/3}\}$. Now, for $i\ge 2$, consider any $p \in \A_{i-2}$ and note that $\abs{h_{i-1}(p)}\le 2n/3$ by assumption. Thus, to each $u \in \A_{i-1}$ with $p(u)=p$, we can allocate $2n/3$ vertices from $N_+(u)$ without creating any conflicting pairs of type (a).\footnote{To keep notation simple, let us consider the case $h_{i-1}(p) = \{u_j :=p+e_j : 1 \le j \le 2n/3\} \subseteq \A_{i-1}$ as an example. Allocating to $u_j$ the vertices $\{u_j + e_k : j \le k \le j+  n/3 -1 \mod 2n/3\}$ and the vertices  $\{u_j + e_k : 2n/3 + 1 \le k \le n\}$ guarantees that there are no conflicts among the $u_j$'s, and thus no conflicting pairs of type (a) overall.} Taking care of conflicts of types (b) and (c) by removing the set $\Gamma(u)$ from the $2n/3$ vertices previously allocated to $u$, we obtain the set $h_i(u)$. 

On the one hand, choosing $\delta>0$ such that $1-\delta = \frac{1+\theta/2}{1+\theta}$, the previous claim ensures that with probability $1-\exp(-\sqrt{n})$, we have $\abs{h_i(u)}\ge (1-\delta) 2n/3$ for every $1 \le i \le (\log(n))^2$ and for every active vertex $u \in \A_{i-1}$. On the other hand, considering any 1-independent, positively associated percolation model on a $(1-\delta)2n/3$-ary tree with marginal $p=(1+\theta)3/2n$, the Paley-Zygmund inequality yields that for some fixed $C \in (0,\infty)$, there are at least $n^c$ active vertices in level $C\log(n)$ with probability at least $\theta/8n$. By taking the difference of these two bounds, \eqref{eq:short-trees-second-attempt} follows.

While the constant $3/2$ in \eqref{eq:short-trees-second-attempt} already comes close to the constant $\sqrt{2}$ in Lemma \ref{TightBranchingProcess}, we note that the allocation function chosen above only considers at most $4n^2/9$ vertices in the second out-neighbourhood of each vertex, i.e. it does not make use of $n^2/18$ vertices in the second out-neighbourhood. This suggests that the argument above is indeed wasteful, and so the rest of this subsection will be devoted to a more subtle treatment of conflicting pairs of type (a).

\subsubsection{Conflicting Pairs of Type (a)}

Before going into the details of the proof, we outline the main idea for resolving conflicts of type (a). We aim to construct allocation functions enabling us to consider $n^2/2$ vertices in the second out-neighbourhood of each vertex instead of only $4n^2/9$ vertices as in the previous approach. Initially, this can easily be achieved without creating conflicting pairs by choosing $h_1(0) = \{e_1,\ldots,e_n\}$ and then allocating $n/2$ vertices to each vertex in $\A_1$ by setting $h_2(e_j) = \{e_j + e_k : j \le k \le j+ n/2 -1 \mod n\}$ whenever $e_j \in \A_1$. Thus, it appears natural to try to allocate about $n$ vertices to active vertices in even levels and about $n/2$ vertices to active vertices in odd levels. However, already in the second level, we are not able to allocate about $n$ vertices to an active vertex $u \in h_2(e_j) \cap \A_2$ if the size of $h_2(e_j) \cap \A_2$ is about $n/2$, which actually happens if the outgoing edges of $e_j$ are `strongly correlated' (for a concrete example, consider the fully correlated case where all outgoing edges of $e_j$ are simultaneously open or closed). 

To deal with this issue, we will always consider two levels together, $L_{i-1}$ and $L_{i}$ for $i$ odd, and decide individually for each vertex $v \in \A_{i-1}$ whether (1) we allocate about $n$ vertices to $v$ and about $n/2$ vertices to its children, or (2) we allocate about $n/2$ vertices to $v$ and about $n$ vertices to its children. The decision for strategy (1) or (2) will depend on the probability distribution of the outgoing edges of $v$. On the one hand, in case the outgoing edges of $v$ are `weakly correlated' (for a concrete example, consider the independent case where all outgoing edges of $v$ are sampled independently), we have $h_{i}(v) \cap \A_{i}$ has size $o(n)$ with high probability, and hence every vertex in $h_{i}(v) \cap \A_{i}$ is only in $o(n)$ conflicting pairs of type (a). Thus, there is no problem in following strategy (1), i.e.\ choosing $h_{i}(v)$ of size about $n/2$ and $h_{i+1}(w)$ of size about $n$ for every $ w \in h_{i}(v) \cap \A_{i}$. On the other hand, in case the outgoing edges of $v$ are `strongly correlated', it might well be that most vertices in $h_{i}(v)$ are active, and so there might be too many conflicting pairs of type (a) in order to allocate to each such vertex about $n$ vertices. In this case, we therefore follow strategy (2). 

However, there is a problem with this decision for strategy (2): It might well be that $v$ is itself in many conflicting pairs $(u,v)$ of type (a), where the outgoing edges of $u$ and of $v$ are `strongly correlated'. Hence, we cannot allocate to each one of them about $n$ vertices. For this reason, we perform a `test sampling' to estimate if for such a vertex $v$, it is likely that many of its outgoing edges are actually open. If it turns out to be likely that many of $v$'s outgoing edges are open, we will call $v$ `expanding' and we indeed follow strategy (2). Otherwise, we return to strategy (1). Importantly, only few vertices will be `expanding', and this will allow us to resolve the described conflicting pairs $(u,v)$. 

We now describe the choice of the allocation functions $(h_i)_{i\ge 1}$ in detail. Fix $\delta>0$ such that $1-\delta = \frac{1+\theta/2}{1+\theta}$. Recall that we have assumed without loss of generality that every edge $e \in \vv{E}(\vv{\Z}^n)$ is open with probability at most $2/n$. 
Let $1 \le i \le (\log(n))^2$ be odd. Our goal is to choose $h_i$ and $h_{i+1}$. Based on Claim 1, it is straightforward to construct 
\begin{equation}
    f_i : \A_{i-1} \to \mathcal R_i
\end{equation}
with $f_i(u) \subseteq N_+(u)$ for every $u \in \A_{i-1}$ such that there are no conflicts of type (b) and (c) and such that with probability $1-\exp(-\sqrt{n})$, $f_i(u)$ has size at least $(1-\delta/2)n$ for every $u \in \A_{i-1}$. Since we are only left with conflicting pairs of type (a), we note that $f_i(u) \cap f_i(v) \neq \emptyset$ only if $p(u) =p(v)$ and in this case $\abs{f_i(u) \cap f_i(v)}=1$. Moreover, it is easy to see that every conflict is associated with a unique conflicting pair.
The function $f_i$ is a precursor of the allocation function $h_i$, and we will proceed by constructing functions $g_i$, $h_i$ such that for every $u \in \A_{i-1}$, we have $f_i(u) \supseteq g_i(u) \supseteq h_i(u)$. 

Next, we will `test sample' a subset of size $\frac{|f_i(u)|}{2 \log(n)}$ from the each $f_i(u)$. We need to make sure that these subsets are disjoint. To this end, whenever $w$ is an element of $f_i(u)$  and $f_i(v)$, we assign $w$ to either $u$ or $v$ uniformly at random, and we denote by $\phi_i(u)$ (resp. $\phi_i(v)$) the vertices assigned to $u$ (resp.\ $v$). We then remove from $\phi_i(u)$ any vertex with probability $\frac{1}{2}$ that is in no other $f_i(v)$. Finally, we pick at random each vertex in $\phi_i(u)$ with probability $\frac{1}{\log(n)}$ and remove the rest. We call the resulting sets $\varphi_i(u)$. We now reveal for every $u \in \A_{i-1}$ the states of the edges $(u,v)$ where $v \in \varphi_i(u)$. If more than $\lambda \abs{\varphi_i(u)}$ of these edges are open, where the constant $\lambda \in (0,1)$ will be specified later, then we call $u$ \emph{expanding}. Otherwise, we call $u$ \emph{non-expanding}.

We now define $g_i : \A_{i-1} \to \mathcal R_i$ by
\begin{equation}
    g_i(u) := f_i(u) \setminus \left(\bigcup_{v \in \A_{i-1}:\, v\neq u} \varphi_i(v) \cup \bigcup_{\substack{v \in \A_{i-1}:\, v\neq u, \\ v\ \text{expanding}}} f_i(v) \right).
\end{equation}
In words, we remove from $f_i(u)$ all vertices that were `test sampled' and all vertices that are contained in some other $f_i(v)$ where $v$ is expanding.\\

\textbf{Claim 2:} Let $1 \le i \le (\log(n))^2$ and let $n$ be sufficiently large. With probability $1- 2\exp(-\sqrt{n})$, it holds that for every active vertex $u \in \A_{i-1}$, the set $g_i(u)$ has size at least $(1-\delta)n$.
\begin{proof}[Proof of Claim 2]
     Fix a vertex $u \in \A_{i-1}$. For any other vertex $v \in \A_{i-1}$, we have $\abs{f_i(u) \cap f_i(v)} = 1$ if $p(u)=p(v)$ and $\abs{f_i(u) \cap f_i(v)} = 0$ otherwise. Now, in the case  $f_i(u) \cap f_i(v) = \{w\}$, the vertex $w$ is only removed from $f_i(u)$ if $v$ is expanding or if $w \in \varphi(v)$. This occurs with probability at most $\frac{1}{\log(n)}$, independently from any other $v' \in \A_{i-1}$.
     A Chernoff bound implies that the probability of removing more that $\delta n / 2$ vertices from $f_i(u)$ is bounded from above by $\exp(-2\sqrt{n})$ for $n$ sufficiently large. Finally, we conclude by taking a union bound over all $u \in \A_{i-1} \subseteq L_{i-1} \cap (\Z_+)^n$ and by using that with probability $1 - \exp(-\sqrt{n})$, $f_i(u)$ has size at least $(1 - \delta/2)n$ for every $u\in \A_{i-1}$.
\end{proof}

We are now in position to define the allocation function $h_i : \A_{i-1} \to \mathcal R_i$. For simplicity, set $n' := (1-\delta)n$. 
First, for every non-expanding vertex $u \in \A_{i-1}$, we start by choosing a random subset of $g_i(u)$ of size $n'$ that we denote by $g'_i(u)$. This is possible with probability $1-2\exp(-\sqrt{n})$ by Claim 2. To obtain $h_i$ from $g'_i$, we remove any conflict $w \in g'_i(u) \cap g'_i(v)$ from either $g'_i(u)$ or $g'_i(v)$ uniformly at random, and we remove any vertex $w \in g'_i(u)$ involved in no conflict from  $g'_i(u)$ with probability $1/2$. The resulting sets $h_i(u)$ have expected size $n'/2$, and by construction, none of the non-expanding vertices is involved in conflicts. 
Second, for every expanding vertex $u \in \A_{i-1}$, we simply choose $g'_i(u)$ to be a random subset of $g_i(u)$ of size $n'$ and define $h_i(u) := g'_i(u)$. This is again possible with probability $1-2\exp(-\sqrt{n})$, and by construction, none of the expanding vertices is involved in conflicts. 
In summary, with probability $1-3\exp(-\sqrt{n})$, we have successfully defined the allocation function $h_i : \A_{i-1} \to \mathcal R_i$.

The next claim establishes that vertices which are wrongly classifying as non-expanding are rare. It is in the proof of this claim that we specify the constant $\lambda \in (0,1)$.\\

\textbf{Claim 3:} Let $1 \le i \le (\log(n))^2$ and let $n$ be sufficiently large. With probability $1- \exp(-\sqrt{n})$, it holds that for every non-expanding vertex $u \in \A_{i-1}$, at most $\delta n / 2$ of the outgoing edges of $u$ are open.
\begin{proof}[Proof of Claim 3]
    Fix a non-expanding vertex $u \in \A_{i-1}$. We write  $\varphi_i^\text{open}(u)$ and $f_i^\text{open}(u)$ for the subset of vertices $v$ in $\varphi_i(u)$ respectively in $f_i(u)$, for which the edge $(u,v)$ is open. We aim to show that for some constant $\lambda \in (0,1)$,
    \begin{equation}\label{eq:claim-3}
        \P\left[ \abs{f_i^\text{open}(u)} \ge \frac{\delta n}{2} \biggm| \abs{\varphi_i^\text{open}(u)} \le \lambda \abs{\varphi_i(u)} \right] \le 2\exp(-2\sqrt{n}).
    \end{equation}
    
    It is thus without loss of generality to assume that $\abs{f_i^\text{open}(u)} \ge n^{3/4}$. Chernoff bounds directly imply that for $\epsilon>0$ and $n$ sufficiently large, we have with probability at least $1 - \exp(-2\sqrt{n})$, 
    \begin{align*}
        (1 - \epsilon)\frac{\abs{f_i(u)}}{2 \log(n)}  \leq &\abs{\varphi_i(v)} \leq (1+\epsilon)\frac{\abs{f_i(u)}}{2 \log(n)}, \\
        (1-\epsilon)\frac{\abs{f_i^\text{open}(u)}}{2 \log(n)} \leq &\abs{\varphi^\text{open}_i(v)} \leq (1+\epsilon)\frac{\abs{f_i^\text{open}(u)}}{2 \log(n)}.
    \end{align*}
    For $\lambda := \frac{\delta(1-\epsilon)}{3(1+\epsilon)}$, we deduce that
    \begin{equation}
        \P\left[ \abs{\varphi_i^\text{open}(u)} \le \lambda \abs{\varphi_i(u)},\  \abs{f_i^\text{open}(u)} \ge \frac{\delta n}{2} \right]
        \le  \exp(-2\sqrt{n}).
    \end{equation}
    When combined with the fact that 
    \begin{align*}
        \P\left[ \abs{\varphi_i^\text{open}(u)} \le \lambda \abs{\varphi_i(u)} \right] &\ge \P\left[  \abs{f_i^\text{open}(u)} \le \frac{\lambda n}{4\log(n)},\ \abs{\varphi_i(u)} \geq \frac{n}{4\log(n)}\right] \\
        &\ge    \left(1-\frac{8\log(n)}{\lambda n}\right) \cdot \left(1 - o(1)\right) = 1 - o(1),
    \end{align*}
    where we used independence and Markov's inequality in the second step, we obtain \eqref{eq:claim-3} for $n$ sufficiently large. Taking a union bound over all non-expanding $u \in \A_{i-1} \subseteq L_{i-1} \cap (\Z_+)^n$ concludes the proof.
\end{proof}

We now define the allocation function $h_{i+1} : \A_{i} \to \mathcal R_{i+1}$. Recall that $n' := (1-\delta)n$. It is again straightforward based on Claim 1 to construct 
\begin{equation}
    f_{i+1} : \A_{i} \to \mathcal R_{i+1}
\end{equation}
with $f_{i+1}(u) \subseteq N_+(u)$ for every $u \in \A_{i}$ such that there are no conflicts of type (b) and (c) and such that with probability $1-\exp(-\sqrt{n})$, $f_{i+1}(u)$ has size at least $(1-\delta/2)n$ for every $u \in \A_{i}$. 
Next, we need to deal with conflicting pairs of type (a) to obtain $h_{i+1}$ from $f_{i+1}$. We proceed separately for each parent $p \in \A_{i-1}$. 
First, if $p$ is expanding, we can deterministically allocate $n/2$ vertices from $N_+(u)$ to each vertex $u \in \A_i$ with parent $p$ so that there are no conflicting pairs of type (a). To obtain $h_{i+1}(u)$ from these $n/2$ vertices, we simply pick at random a subset of size $n'/2$ of vertices that also belong to $f_{i+1}(u)$. This is possible with probability $1-\exp(-\sqrt{n})$. 
Second, if $p$ is non-expanding, Claim 3 guarantees that with probability $1- \exp(-\sqrt{n})$, there are at most $\delta n / 2$ vertices $u \in \A_i$ with parent $p$. Thus, we can obtain $h_{i+1}(u)$ from $f_{i+1}(u)$ by removing all remaining conflicts (which are necessarily of type (a)) and then picking at random a subset of size $n'$. This is possible with probability $1-2\exp(-\sqrt{n})$.   
In summary, with probability $1-2\exp(-\sqrt{n})$, we have successfully defined the allocation function $h_{i+1} : \A_{i} \to \mathcal R_{i+1}$. By construction, for every $u \in \A_i$, we have
$\abs{h_{i+1}(u)} = n'$ if and only if $\mathrm E[\abs{h_{i}(p(u))}] = n'/2$, and $\abs{h_{i+1}(u)} = n'/2$ if and only if $\abs{h_{i}(p(u))} = n'$.

\subsubsection{Conclusion}

    Denote by $\mathcal{H}_i$ the high-probability event that the allocation function $h_i$ could successfully be defined given $\mathcal A_{i-1}$. By convention, we set $\mathcal A_i := \emptyset$ on the complement of $\mathcal{H}_i$. 
    Our goal is to prove that for every $1\le i \le \log(n)^2$ odd,
    \begin{equation}\label{eq:conclusion}
        \P\left[\abs{\A_{i+1}} \ge \frac{1}{4} (1+\theta/2)^{i+1}\right] \geq \frac{\theta}{40n^2}.
    \end{equation}
    Once we have established this lower bound, the proof of Lemma \ref{TightBranchingProcess} will be complete.

    To establish \eqref{eq:conclusion}, we compute the first and second moment of $\abs{\A_{i+1}}$ and then apply the Paley-Zygmund inequality. Recall that we have assumed without loss of generality that each edge $e \in \vv{E}(\vv{\Z}^n)$ is open with probability $p=(1 + \theta)\frac{\sqrt{2}}{n}$. To shorten notation, we define $\varphi := (1 + \frac{\theta}{2})$, and we introduce the random vector $\tau \in \{0,1\}^{\A_{i-1}}$, where for any vertex $u \in \A_{i-1}$ we have that $\tau_u = 1$ if $u$ is expanding and $\tau_u = 0$ otherwise. Moreover, for $u \in L_{i-1} \cap (\Z_+)^n$, we define 
    \begin{equation}
        c_1(u) := \abs{h_i(u) \cap \A_i}\quad \text{and} \quad c_2(u) := \abs{\bigcup_{v \in h_i(u) \cap \A_i} h_{i+1}(v) \cap \A_{i+1}}, 
    \end{equation}
    which are the number of active children and grandchildren, respectively. By convention, $c_1(u)=c_2(u)=0$ if $u \notin \A_{i-1}$.
    
    We begin with computing the first moment of $\abs{\A_{i+1}}$. For every $u \in L_{i-1} \cap (\Z_+)^n$,
    \begin{align} 
        &\mathrm E[c_2(u) \mid \A_{i-1}] = \left(\mathrm E[c_2(u) \mathbf{1}_{\tau_u = 0} \mathbf{1}_{\mathcal{H}_i \cap \mathcal{H}_{i+1}}]  + \mathrm E[c_2(u) \mathbf{1}_{\tau_u = 1} \mathbf{1}_{\mathcal{H}_i \cap \mathcal{H}_{i+1}}] \right) \cdot \mathbf{1}_{u \in \A_{i-1}} \\
        &= \left( (n'p) \cdot \mathrm E[ c_1(u) \mathbf{1}_{\tau_u = 0} \mathbf{1}_{\mathcal{H}_i \cap \mathcal{H}_{i+1}} ]  + \frac{n'p}{2} \cdot \mathrm E[c_1(u) \mathbf{1}_{\tau_u = 1} \mathbf{1}_{\mathcal{H}_i \cap \mathcal{H}_{i+1}} ] \right) \cdot \mathbf{1}_{u \in \A_{i-1}} \\
        &= \left(\frac{n'p}{2}  \cdot \mathrm E[ \abs{g'_i(u) \cap \A_i} \mathbf{1}_{\tau_u = 0} \mathbf{1}_{\mathcal{H}_i \cap \mathcal{H}_{i+1}} ]  + \frac{n'p}{2} \cdot \mathrm E[\abs{g'_i(u) \cap \A_i} \mathbf{1}_{\tau_u = 1} \mathbf{1}_{\mathcal{H}_i \cap \mathcal{H}_{i+1}} ] \right) \cdot \mathbf{1}_{u \in \A_{i-1}} \\ 
        &= \frac{n'p}{2}  \cdot \mathrm E[ \abs{g'_i(u) \cap \A_i} \mathbf{1}_{\mathcal{H}_i \cap \mathcal{H}_{i+1}}] \cdot \mathbf{1}_{u \in \A_{i-1}}  = \frac{(n'p)^2}{2}  \cdot \mathrm E[\mathbf{1}_{\mathcal{H}_i \cap \mathcal{H}_{i+1}} ] \cdot \mathbf{1}_{u \in \A_{i-1}}  \\
        & \ge \varphi^2 \cdot (1-5\exp(-\sqrt{n})) \cdot  \mathbf{1}_{u \in \A_{i-1}} .
    \end{align}
    It now follows by induction that 
    \begin{equation}
        \varphi^{i+1} \ge \mathrm E[ \abs{\A_{i+1}}] \ge \varphi^{i+1} \cdot (1-\exp(-n^{1/3})) .
    \end{equation}

    To compute the second moment, let $u\neq v \in L_{i-1} \cap (\Z_+)^n$. To shorten notation, we write $\mathrm E'[\ \boldsymbol{\cdot}\ ]$ for  $\mathrm E[\ \boldsymbol{\cdot}\ \mathbf{1}_{\mathcal{H}_i \cap \mathcal{H}_{i+1}} ]$.  As in the first moment computation, we obtain
    \begin{align}
        &\mathrm E[c_2(u) c_2(v) \mid \A_{i-1}] = \sum_{x,y\in \{0,1\}} \mathrm E'[c_2(u) c_2(v) \mathbf{1}_{(\tau_u,\tau_v)=(x,y)}  ] \cdot  \mathbf{1}_{u,v \in \A_{i-1}} \\
        &=  \sum_{x,y\in \{0,1\}} \frac{(n'p)^2}{2^{x+y}}\cdot \mathrm E'[c_1(u) c_1(v) \mathbf{1}_{(\tau_u,\tau_v)=(x,y)}   ] \cdot  \mathbf{1}_{u,v \in \A_{i-1}}.
    \end{align}
    If $p(u) \neq p(v)$, we have
    \begin{equation}
        \mathrm E'[c_1(u) c_1(v) \mathbf{1}_{(\tau_u,\tau_v)=(x,y)}] =  \mathrm E'[c_1(u) \mathbf{1}_{\tau_u=x}] \cdot \mathrm E'[c_1(v) \mathbf{1}_{\tau_v=y}],
    \end{equation}
    and similar to the first moment computation, one then obtains the upper bound  
    \begin{align}
        \mathrm E'[c_2(u) c_2(v) \mid \A_{i-1}] \le \varphi^4 \cdot  \mathbf{1}_{u,v \in \A_{i-1}}.
    \end{align}
    If $p(u) = p(v)$, we have 
    \begin{equation}
        \mathrm E'[c_1(u) c_1(v) \mathbf{1}_{(\tau_u,\tau_v)=(x,y)}] \le  (\mathrm E'[c_1(u) \mathbf{1}_{\tau_u=x}]+1) \cdot (\mathrm E'[c_1(v) \mathbf{1}_{\tau_v=y}]+1)
    \end{equation}
    since the state of $\tau_v$ (resp.\ $\tau_u$) affects at most one vertex in the assignment $g_i(u)$ (resp.\ $g_i(v)$). Thus, we get
    \begin{align}
        &\mathrm E[c_2(u) c_2(v) \mid \A_{i-1}] \\
        &\le  \sum_{x\in \{0,1\}} \frac{n'p}{2^{x}}\cdot  (\mathrm E'[c_1(u) \mathbf{1}_{\tau_u=x}]+1) \cdot \sum_{y\in \{0,1\}} \frac{n'p}{2^{y}} \cdot (\mathrm E'[c_1(v) \mathbf{1}_{\tau_v=y}]+1) \cdot\mathbf{1}_{u,v \in \A_{i-1}}  \\
         &\le \left(\frac{n'p}{2}\right)^2 \cdot  \sum_{x\in \{0,1\}}  \left(\mathrm E'[ \abs{g'_i(u) \cap \A_i} \mathbf{1}_{\tau_u=x}]+2\right)  \cdot  \sum_{y\in \{0,1\}}  \left(\mathrm E'[ \abs{g'_i(v) \cap \A_i} \mathbf{1}_{\tau_v=y}]+2 \right) \cdot \mathbf{1}_{u,v \in \A_{i-1}}  \\
         &\le \left(\frac{n'p}{2}\right)^2 \cdot  (n'p + 4)^2 \cdot \mathbf{1}_{u,v \in \A_{i-1}}  \le \varphi^2 (\varphi+4)^2\cdot \mathbf{1}_{u,v \in \A_{i-1}} ,
    \end{align}
    Using both upper bounds, we obtain
    \begin{align}
        &\mathrm E[ \abs{\A_{i+1}}^2 \mid \A_{i-1}] = \sum_{u,v \in \A_{i-1}} \mathrm E[ c_2(u) c_2(v) \mid \A_{i-1}] \\
         &= \sum_{u \in \A_{i-1}} \mathrm E[ c_2(u)^2 \mid \A_{i-1}] +  \sum_{\substack{u\neq v \in \A_{i-1}: \\ p(u)=p(v)}} \mathrm E[ c_2(u) c_2(v) \mid \A_{i-1}] + \sum_{\substack{u\neq v \in \A_{i-1}: \\ p(u)\neq p(v)}} \mathrm E[ c_2(u) c_2(v) \mid \A_{i-1}]\\        
         &\le \varphi^2 \cdot n^2 \cdot \abs{\A_{i-1}} +  \varphi^2 (\varphi+4)^2 \cdot n\cdot  \abs{\A_{i-1}}+ \varphi^4 \cdot  \abs{\A_{i-1}}^2 \le 2\varphi^2 \cdot n^2 \cdot \abs{\A_{i-1}} + \varphi^4 \cdot \abs{\A_{i-1}}^2,
    \end{align}
    where in the first inequality, we used the simple upper bound $\mathrm E'[ c_2(u)^2 \mid \A_{i-1}] \le n^2 \cdot \mathrm E'[ c_2(u) \mid \A_{i-1}]$ and that there are at most $n$ vertices in $\A_{i-1}$ sharing the same parent, and in the second inequality we used that $n$ is sufficiently large. 
    Using this recursive formula and the first moment computation, one obtains the upper bound
    \begin{equation}
        \mathrm E[ \abs{\A_{i+1}}^2] \leq  2n^2 \varphi^{i+1} \frac{\varphi^{i+3}-1}{\varphi^2-1}+ \varphi^{2(i+1)} \le \frac{4n^2\varphi^2}{\varphi^2-1} \cdot \varphi^{2(i+1)} .
    \end{equation}
    Finally, we make use of the Paley–Zygmund inequality to get
    \begin{equation}
        \P\left[\abs{\A_{i+1}} \ge \frac{1}{2} \mathrm E[ \abs{\A_{i+1}}] \right] \geq \frac{1-\varphi^{-2}}{16 n^2}  (1-\exp(-n^{1/3}))^2 \ge \frac{\theta}{40 n^2}.
    \end{equation}
    This establishes \eqref{eq:conclusion} and thereby concludes the proof of Lemma \ref{TightBranchingProcess}.

\subsection{Connecting Short Trees: Upper  Bound of Theorem \ref{thm:oriented-d-large}}\label{Sec:Connecting}

We consider  $\vv{\Z}^{2+n} = \vv{\Z}^{2} \times \vv{\Z}^{n}$ and introduce the directed subgraphs $$H_{a,b} = \{v \in \vv{\Z}^{2} \times \vv{\Z}^{n}: v_1 = a, v_2 = b\}.$$ This way, we may think of $\vv{\Z}^{2} \times \vv{\Z}^{n}$ as $\vv{\Z}^{2}$, where each vertex comes with its own copy of $\vv{\Z}^{n}$.  The basic idea is to start with some fixed vertices in $H_{0,0}$ and then grow the open cluster inside of $H_{0,0}$ according to Lemma \ref{TightBranchingProcess} till size $n^{c}$, for some $c > 0$, and then extend it to $H_{1,0}$ and to $H_{0,1}$. The extensions to $H_{1,0}$ and to $H_{0,1}$, respectively, will then be the seeds of new trees inside of $H_{1,0}$ and of $H_{0,1}$. This way the process is continued.

 Fix $\theta \in (0,1/4)$ and let $\P \in \mathcal{P}_a(\vv{\Z}^{2} \times \vv{\Z}^{n}, \sqrt{2} \cdot \frac{(1 + 3\theta)}{n + 2})$.  Our goal is to show that $\P$ percolates for $n$ sufficienctly large. Throughout the proof, we fix the strict level structure $L$ given by the diagonals $L_{i} = \{v \in \vv{\Z}^{2} \times \vv{\Z}^{n}: v_1 + \ldots v_{n+2} = i\}$. According to Lemma \ref{lem:levelwise-decoupling}, we may assume without loss of generality that $\P$ is levelwise independent.
 According to Lemma \ref{TightBranchingProcess}, let us choose $\ell = \ell(n)$  such that for any $n' \ge (1-\theta)n$ and $\mathrm Q \in \mathcal{P}_a(\vv{\Z}^{n'},\sqrt{2} \cdot \frac{(1 + 3\theta)}{n + 2})$, we have $\mathrm Q[\abs{\vv{\mathcal{C}}_0(\omega) \cap L_{\ell - 4}} \geq n^{9}] \geq \frac{\theta}{40 n^2}$ for $n$ sufficiently large.
 
    \textbf{Claim:} For any $(a,b) \in \vv{\Z}^2$, $k \in \Z$, and $U \subseteq H_{a,b} \cap L_k$ of size at least $n^7$, 
    $$\P\left[|\vv{\mathcal{C}}_{U}(\omega) \cap H_{a+1,b} \cap L_{k + \ell}| \geq n^7\right] \ge 1 - \epsilon(n),$$
    where $\epsilon(n) \xrightarrow{n \to \infty}0$, uniformly in the choice of $(a,b)$, $k$, and $U$. The analogous statement holds true  for $H_{a,b+1}$ instead of $H_{a+1,b}$.
    \begin{proof}
    It is without loss of generality to assume that $U$ has size $n^7$. First, we argue that with probability $1-o(1)$, we have that for every $(a,b,u) \in U$, there exists a vertex $(a,b,v) \in N_+^3((a,b,u))$, i.e. reachable by a directed path of length $3$, such that 
        \begin{equation}\label{eq:thm-oriented-higher-1}
            |\vv{\mathcal{C}}_{(a,b,v)}(\omega) \cap H_{a,b} \cap L_{k + \ell - 1}| \geq n^{9}.
        \end{equation}        
    Indeed, fix $(a,b,u) \in U$ and denote by $N \subseteq H_{a,b}$ the set of all vertices  $(a,b,v) \in N_+^3((a,b,u))$ satisfying $u_j = v_j$ for every $j>\frac{n}{\log(n)}$. Note that $\abs{N} = \Omega \big( (\frac{n}{\log(n)})^3 \big)$ and that the sets 
        \begin{equation}
            \left(Z(a,b,v) := \left\{(a,b,w) : w_j = v_j, \forall j \le \frac{n}{\log(n)}; w_j \ge v_j, \forall j > \frac{n}{\log(n)}\right\}\right)_{(a,b,v) \in N}
        \end{equation}
    are pairwise disjoint. By Lemma \ref{TightBranchingProcess} applied to the subgraph $Z(a,b,v)$, whose vertices have outdegree $(1-1/\log(n))n$, we have that the event in \eqref{eq:thm-oriented-higher-1} restricted to $Z(a,b,v)$ occurs with probability at least $\theta/(40 n^2)$, independently for each $(a,b,v) \in N$. Thus, the event occurs for some $(a,b,v) \in N$ with probability at least $1-\exp(-\gamma \sqrt{n})$ for some $\gamma >0$, and taking a union bound over $U$ completes the first step. 

    Second, let us denote by $A$ the event that for every $(a,b,u) \in U$, there exists a vertex $(a,b,v) \in N_+^3((a,b,u))$ such that 
        \begin{equation}\label{eq:thm-oriented-higher-2}
            |\vv{\mathcal{C}}_{(a,b,v)}(\omega) \cap H_{a+1,b} \cap L_{k + \ell}| \geq n^{7}.
        \end{equation}  
    It easily follows from \eqref{eq:thm-oriented-higher-1} that $\P[A] = 1-o(1)$.
    Indeed, it suffices that for at least $n^7$ of the $n^{9}$ vertices in $\vv{\mathcal{C}}_{(a,b,v)}(\omega) \cap H_{a,b} \cap L_{k + \ell - 1}$, the edge $((a,b,w),(a+1,b,w))$ is open. This happens with probability $1-o(1)$ since the marginals are of order $1/n$.

    Finally, it remains to argue that with probability $1-o(1)$, there exists some $(a,b,u) \in U$ which is connected by an open path to a vertex $(a,b,v) \in N_+^3((a,b,u))$ satisfying \eqref{eq:thm-oriented-higher-2}. To this end, we sample the edges between levels $L_{k+3}$ and $L_{k+\ell}$ conditional on the event $A$, and we pick for each $(a,b,u) \in U$ a path of length $3$ to some $(a,b,v) \in N_+^3((a,b,u))$ satisfying \eqref{eq:thm-oriented-higher-2}. Since these paths are of length $3$, any path can share a vertex with at most $4n^3$ other paths, and so the initial collection of $n^7$ paths contains a subcollection of $n^4 / 4$ paths which are pairwise disjoint. 
    Since each path is open with probability at least $1/n^3$, we conclude that at least one of the paths is open with probability $1-o(1)$.
    \end{proof}

Fix an arbitrary set $O \subseteq H_{0,0} \cap L_0$ consisting of $n^7$ vertices. For any percolation configuration $\omega \in \{0,1\}^{\vv{\Z}^2 \times \vv{\Z}^n}$ and $i\ge 0$, we define 
    \begin{equation}
        \mathcal{Y}_i(\omega) := \left\{(a,b) \in \vv{\Z}^2 : a+b = i\ \text{and}\ |\vv{\mathcal{C}}_O(\omega) \cap H_{a,b} \cap L_{i \cdot h}| \geq n^7\right\}.
    \end{equation}
Note that the choice of $O$ guarantees $\mathcal{Y}_0(\omega) = \{(0,0)\}$ and that $\mathcal{Y}_i$ is increasing in $\omega$.
According to the previously established claim, each of the events $\{(1,0) \in \mathcal{Y}_1(\omega)\}$ and $\{(0,1) \in \mathcal{Y}_1(\omega)\}$ has probability $1-\epsilon(n)$ under $\P$, and the two events are positively correlated. 
This being said, it will come as no surprise that 
    \begin{equation}\label{eq:thm-oriented-higher-3}
        \P[\mathcal{Y}_i(\omega) \neq \emptyset,\ \forall i\ge 0] > 0.
    \end{equation}
To establish \eqref{eq:thm-oriented-higher-3} rigorously, one compares the stochastic process $(\mathcal{Y}_i)_{i\ge 0}$ to the Markov chain $(\mathcal{X}_i)_{i\ge 0}$ introduced in Subsection \ref{sec:StochasticDominationOrientedClusters}  with initial state $\{(0,0)\}$ and parameter $p= 1 - \epsilon(n)$. 
As in the proofs of Lemma \ref{lem:levelwise-decoupling}, \ref{lem:markov-chain-lemma-1} and \ref{lem:markov-chain-lemma-2}, the result of Kamae, Krengel and O'Brien \cite{Kamae1977} allows to show that $\mathcal{X}_i \preceq \mathcal{Y}_i$ for every $i\ge 0$, and we leave the details of this step to the reader. If $(1-\epsilon(n))^2 > p_c^{\mathrm{site}}(\vv{\Z}^2)$, Lemma \ref{lem:markov-chain-lemma-2} then implies survival of the Markov chain $(\mathcal{X}_i)_{i\ge 0}$, meaning that the event $\{\mathcal{X}_i\neq \emptyset, \forall i\ge 0\}$ has positive probability, and this establishes \eqref{eq:thm-oriented-higher-3} for $n$ sufficiently large. It now follows by construction that $\P$ percolates and this concludes the proof of Theorem \ref{thm:oriented-d-large}.

\section{Directions for Future Research}\label{sec:open-questions}

Theorem \ref{thm:trees} implies the equality $p_{a}^{+}(T)=p_{c}(T)$ for trees, whereas Theorem \ref{thm:unoriented-d-2} establishes the strict inequality $p_{a}^{+}(\Z^2) > p_{c}(\Z^2)$ on the square lattice. It is natural to ask if the strict inequality
\begin{equation}
    p_{a}^{+}(G) > p_{c}(G)
\end{equation}
holds for every transitive graph $G$ that is not a tree.

In Theorem \ref{thm:trees}, positive association and finite-range dependence are necessary assumptions. However, it would be interesting to try to extend Theorem \ref{thm:trees} to classes of positively associated percolation models satisfying a sufficiently fast decay of correlations.

In the high-dimensional case, we were able to identify the first-order term of $p_a^+(\vv{\Z}^n)$ as $n \to \infty$. However, we were only able to show that 
\begin{equation}
   \frac{1}{2n} \cdot (1+o(1)) \le  p_a^+(\Z^n) \le \frac{\sqrt{2}}{n} \cdot (1+o(1)) \quad \text{as}\ n \to \infty.
\end{equation}
This leaves open the question of exactly identifying the first-order term of $p_a^+(\Z^n)$.

In the two-dimensional case, the lower and upper bounds in Theorems \ref{thm:unoriented-d-2} and \ref{thm:oriented-d-2} remain far from each other. Conjectures for the correct values or significant improvements of the rigorous bounds would be very interesting and could lead to new methods. 

A natural subclass of $\mathcal{P}_a(G,p)$ is obtained by restricting to invariant percolation models. More precisely, for a transitive graph $G$, consider the class $\mathcal{P}_{a,s}(G,p)$ of positively associated, finite-range dependent percolation models with marginals $p$ that are invariant under the symmetries of $G$ (or under some subgroup of the symmetry group). Introducing 
\begin{equation}
	p_{a,s}^{+}(G) := \sup\{p : \exists\; \P \in \mathcal{P}_{a,s}(G,p) \text{ that does not percolate}\},
\end{equation}
it is clear that $p_{a,s}^{+}(G) \le p_a^{+}(G)$. One direction would then be to make use of the symmetries to establish better bounds on $p_{a,s}^{+}(G)$. Another interesting question would be to study whether $p_{a,s}^{+}(G) > p_c(G)$ or not. In the case of $\Z^n$, we conjecture that 
\begin{equation}
   p_{a,s}^+(\Z^n) = \frac{1}{2n} \cdot (1+o(1)) \quad \text{as}\ n \to \infty.
\end{equation}

\small
\bibliographystyle{alpha}
\bibliography{refs}

\newpage

\appendix 
\section{Appendix} \label{appendix}

It remains to describe the algorithm used to compute the survival probabilities in Figure \ref{fig:Survival_Probabilities}. The algorithm is similar to the one used in \cite{Balister1993}. We consider the directed graph $\vv{\mathcal{B}}_1(u;w,\ell)$ with the strict level structure introduced below \eqref{eq:definition-diagonal-box} and compute by induction the probability distribution of $\mathcal{X}_i \subseteq L_i$ for every $1\le i \le 2(\ell+w)$. Initializing with $\mathcal{X}_0 = L_0$ then yields the desired survival probabilities for $\ell = 0$ respectively $\ell = w+1$. Below, we present the induction step. As even levels have size $w+1$ and odd levels have size $w$, there is a slight difference between $i$ even and $i$ odd. We first consider the case, where $i$ is even and then discuss what changes are needed if $i$ is odd.
\begin{algorithm}[H]
\caption{Algorithm computing the distribution of $\mathcal{X}_i^{t+1}$ from the distribution of $\mathcal{X}_i^{t}$ in the case that $i$ is even}\label{alg:Computation}
\begin{algorithmic}
\State Initialize $P[\mathcal{X}_i^{t+1} = c] = 0,$ for all $c = (c_j)_{j=0}^{w+1} \in \{0,1\}^{w + 1}\times \{\ell, r\}$.
\For{$c \in \{0,1\}^{w + 1}\times \{\ell, r\}$}
\If{$c_{w+1} = \ell$ or $c_{t} = 0$}
\If{$c_{t+1} = 1$}
\State Set $d=c$ and replace in $d$, $d_{w+1}$ by $r$ and $d_{t}$ by $1$
\State Update: $P[\mathcal{X}_i^{t+1} = d] = P[\mathcal{X}_i^{t+1} = d] + p \cdot P[\mathcal{X}_i^{t} = c]$
\State 
\State Set $d=c$ and replace in $d, d_{w+1}$ by $\ell$ and $d_{t}$ by $0$
\State Update: $P[\mathcal{X}_i^{t+1} = d] = P[\mathcal{X}_i^{t+1} = d] + (1-p) \cdot P[\mathcal{X}_i^{t} = c]$
\ElsIf{$c_{t+1} = 0$}
\State Set $d=c$ and replace in $d, d_{w+1}$ by $\ell$ and $d_{t}$ by $0$
\State Update: $P[\mathcal{X}_i^{t+1} = d] = P[\mathcal{X}_i^{t+1} = d] + P[\mathcal{X}_i^{t} = c]$
\EndIf
\ElsIf{$c_{w+1} = r$ and $c_{t} = 1$}
\State Set $d=c$ and replace in $d, d_{w+1}$ by $r$ and $d_{t}$ by $1$
\State Update: $P[\mathcal{X}_i^{t+1} = d] = P[\mathcal{X}_i^{t+1} = d] + p \cdot P[\mathcal{X}_i^{t} = c]$
\State
\State Set $d=c$ and replace in $d, d_{w+1}$ by $r$ and $d_{t}$ by $0$
\State Update: $P[\mathcal{X}_i^{t+1} = d] = P[\mathcal{X}_i^{t+1} = d] + (1-p) \cdot P[\mathcal{X}_i^{t} = c]$
\EndIf
\EndFor
\end{algorithmic}
\end{algorithm}
To keep notation simple, we rotate the diagonal box $\vv{\mathcal{B}}_1(u;w,\ell)$ counter-clockwise by $45^{\circ}$ and enumerate the vertices of the box by row and column as $(i,j),$ where in the even rows we have $w + 1$ vertices and in the odd rows we have $w$ vertices.\footnote{In the algorithm, we start the enumeration of the columns with $j=0$, and so the first $w$ positions of $\mathcal{X}_i^w$ are referred to as $(i,0),\ldots,(i,w-1)$.} Let us denote by $\mathcal{X}_i^{t} \in \{0,1\}^{w + 1},$ where $0 \leq t \leq w$, the process where the $s$-th, $0 \leq s < t$, position is $1$ if the vertex $(i+1, s)$ is in $\mathcal{X}_{i+1}$ and the $s$-th, $s \geq t$, position is $1$ if the vertex $(i, s)$ is in $\mathcal{X}_{i}$. From the definition, it is clear that $\mathcal{X}_i^{0}$ equals $\mathcal{X}_i$ and the first $w$ positions of $\mathcal{X}_i^{w}$ are equal to $\mathcal{X}_{i+1}$. It is clear that if we can compute the distribution of $\mathcal{X}_i^{t+1}$ given the distribution of $\mathcal{X}_i^{t}$, we have an algorithm for computing the distribution of $\mathcal{X}_{i+1}$ from the distribution of $\mathcal{X}_{i}.$ But there is one information lacking. Indeed, in order to generate $\mathcal{X}_i^{t+1}$ from $\mathcal{X}_i^{t}$, we need to know whether for the current interval $W$, we have already added a vertex to $N^+(W)$ or not (see the definition of $(\mathcal X_i)_{i\ge 0}$ in Subsection \ref{sec:StochasticDominationOrientedClusters}). More formally, if there is some vertex $(i+1, s_1) \in \mathcal{X}_{i+1}$ with $s_1 < t$ such that for all $s_2 \in (s_1,t)$ , we have $(i, s_2) \in \mathcal{X}_i$, then we call  $\mathcal{X}_i^t$ in state $r$. Otherwise, we call it in state $\ell$. With this extra information, it is now easy to compute the distribution of $\mathcal{X}_i^{t+1}$ from the distribution of $\mathcal{X}_i^{t}.$ By abuse of notation, we thus consider $\mathcal{X}_i^t \in \{0,1\}^{w+1} \times \{\ell, r\}$, and $\mathcal{X}^0_i$ is actually equal to $(\mathcal{X}_i , r)$. With these preparations it is now fairly easy to compute the distribution of $\mathcal{X}_i^{t+1}$ given the distribution of $\mathcal{X}_i^{t}$ as presented in Algorithm \ref{alg:Computation}.

In the case that $i$ is odd, we simply define $\mathcal{X}_i^0$ to be $(0,\mathcal{X}_i, r)$, i.e.\ we set the first position to be 0 and the second to ($w+1$)'th positions  according to $\mathcal{X}_i$, and compute $\mathcal{X}^t_i$ for all $1 \leq t \leq w+1,$ where the $w+1$ first positions of $\mathcal{X}^{w+1}_i$ are then equal to $\mathcal{X}_{i+1}$. We have to be slightly careful when computing $\mathcal{X}_i^{w+1}$ from $\mathcal{X}_i^{w}$ since the vertex $(i+1,w)$ can only be reached through the right outgoing edge of $(i,w-1)$. Hence, for $t=w$, we have to change the first if-clause to read:
\begin{algorithm}[H]
    \begin{algorithmic}
    \If{$c_{w+1} = \ell$ or $c_{t} = 0$}
        \State Replace in $d, d_{w+1}$ by $\ell$ and $d_{t}$ by $0$
        \State Update: $P[\mathcal{X}_i^{t+1} = d] = P[\mathcal{X}_i^{t+1} = d] + P[\mathcal{X}_i^{t} = c]$
    \EndIf
\end{algorithmic}
\end{algorithm}
It is easy to see that the computation of the survival probabilities has time complexity $ O(w \ell \cdot 2^{w}).$

\end{document}